\documentclass{elsarticle}
\usepackage{graphicx} 
\usepackage{soul}
\usepackage{mathtools}
\usepackage[hyphens]{url}
\usepackage{amsmath,amsthm,amscd,url,enumerate,xcolor,amssymb}
\usepackage[colorlinks=true]{hyperref}
\usepackage{xcolor}
\usepackage{subfiles}
\usepackage{tikz}
\usepackage{multicol}
\usepackage{subcaption}
\usepackage[ruled,vlined,linesnumbered]{algorithm2e}
\usepackage{algpseudocode}
\usepackage[inline]{enumitem}
\usepackage{comment}
\usepackage{makecell}
\usepackage{footnote}
\usepackage{minted}
\usepackage{cleveref}

\SetArgSty{upshape}
\SetEndCharOfAlgoLine{}
\DecMargin{2.2mm}
\SetKwInput{Input}{Input}
\SetKwInput{Output}{Output}
\SetKwFor{While}{While}{do}{}
\SetKwFor{ForEach}{For each}{do}{}
\SetKwIF{If}{ElseIf}{Else}{if}{then}{else if}{else}{endif}
\SetKw{Return}{return}
\SetKw{KwRet}{return}
\SetKw{Recurse}{Recurse}
\SetKwFor{For}{for}{do}{}
\SetKw{Break}{break}

\newtheorem{theorem}{Theorem}
\newtheorem{proposition}[theorem]{Proposition}
\newtheorem{corollary}[theorem]{Corollary}
\newtheorem{definition}[theorem]{Definition}

\theoremstyle{definition}
\newtheorem{example}[theorem]{Example}
\newtheorem*{notation}{Notation}

\newcommand{\End}{\operatorname{End}}
\newcommand{\Aut}{\operatorname{Aut}}
\newcommand{\Vol}{\operatorname{Vol}}

\newcommand{\cond}{\operatorname{cond}}
\newcommand{\disc}{\operatorname{disc}}
\newcommand{\Tr}{\operatorname{Tr}}
\newcommand{\Fpbar}{\overline{\mathbb{F}}_p}
\newcommand{\SSOpr}{{\operatorname{SS}_\mathcal{O}^{\text{pr}}}}

\theoremstyle{plain}
\newtheorem{remark}[theorem]{Remark}

\newtheorem{lemma}[theorem]{Lemma}

\newcommand{\ZZ}{\mathbb{Z}}
\newcommand{\QQ}{\mathbb{Q}}
\newcommand{\Fp}{\mathbb{F}_p}
\newcommand{\treestar}{\ensuremath{\operatorname{tree}^\blacklozenge}}

\newif\ifcomments
\commentstrue

\ifcomments
\newcommand{\SA}[1]{\textcolor{violet}{{\sf (Sarah:} {\sl{#1})}}}
\newcommand{\RB}[1]{\textcolor{red}{{\sf (Ross:} {\sl{#1})}}}
\newcommand{\travis}[1]{\textcolor{orange}{{\sf (Travis:} {\sl{#1})}}}
\newcommand{\jason}[1]{\textcolor{green!70!black}{{\sf (Jason:} {\sl{#1})}}}
\newcommand{\JC}[1]{\textcolor{magenta}{{\sf (James C:} {\sl{#1})}}}
\newcommand{\WG}[1]{\textcolor{olive}{{\sf (Wissam:} {\sl{#1})}}}
\newcommand{\km}[1]{\textcolor{blue}{{\sf (Krystal:} {\sl{#1})}}}

\else
\newcommand{\RB}[1]{}
\newcommand{\SA}[1]{}
\newcommand{\travis}[1]{}
\newcommand{\jason}[1]{}
\newcommand{\JC}[1]{}
\newcommand{\WG}[1]{}
\newcommand{\km}[1]{}

\fi

\title{\texorpdfstring{Cycles and Cuts in Supersingular \(L\)-Isogeny Graphs}{Cycles and Cuts in Supersingular L-Isogeny Graphs}}

\author[vtmath]{Sarah Arpin}
\affiliation[vtmath]{organization={Department of Mathematics, Virginia Tech},
            city={Blacksburg},
            state={Virginia},
            country={USA}}

\author[bristolcs]{Ross Bowden}

\author[bristolcs]{James Clements}

\affiliation[bristolcs]{organization = {School of Computer Science, University of Bristol},
city={Bristol},
country={United Kingdom}
}

\author[ucfmath]{Wissam Ghantous}
\affiliation[ucfmath]{organization = {Department of Mathematics, University of Central Florida},
city={Orlando},
state={Florida},
country = {USA}
}

\author[vtmath]{Jason T.~LeGrow}

\author[uvmcs]{Krystal Maughan}
\affiliation[uvmcs]{organization = {Department of Computer Science, University of Vermont},
city = {Burlington},
state = {Vermont},
country = {USA}
}

\begin{document}


\begin{abstract}
Supersingular elliptic curve isogeny graphs underlie isogeny-based cryptography. For isogenies of a single prime degree $\ell$, their structure has been investigated graph-theoretically.
We generalise the notion of $\ell$-isogeny graphs to $L$-isogeny graphs (studied in the prime field case by Delfs and Galbraith), where $L$ is a set of small primes dictating the allowed isogeny degrees in the graph. We analyse the graph-theoretic structure of $L$-isogeny graphs. Our approaches may be put into two categories: cycles and graph cuts. 

On the topic of cycles, we provide: a count for the number of cycles in the $L$-isogeny graph with cyclic kernels using traces of Brandt matrices; an efficiently computable estimate based on this approach; and a third ideal-theoretic count for a certain subclass of $L$-isogeny cycles.
We provide code to compute each of these three counts.   

On the topic of graph cuts, we compare several algorithms to compute graph cuts which minimise a measure called the \textit{edge expansion}, outlining a cryptographic motivation for doing so. Our results show that a \emph{greedy neighbour} algorithm out-performs standard spectral algorithms for computing optimal graph cuts. We provide code and study explicit examples.

Furthermore, we describe several directions of active and future research. 
\end{abstract}


\maketitle

\section{Introduction}
Isogeny-based cryptography has spurred significant research into supersingular elliptic curve isogeny graphs. Research advancements on the underlying structures are increasingly critical as cryptographic protocols evolve. Supersingular elliptic curve $\ell$-isogeny graphs are $(\ell+1)$-regular, Ramanujan \cite{Pizer80} and individual isogeny steps in the graph have a low computational cost, making these graphs desirable for cryptographic applications. The first supersingular isogeny-based cryptographic protocol \cite{CGL09} was developed quite recently, and mathematicians have only been studying these graphs with cryptographic applications in mind for the past two decades.

In this work, we consider supersingular elliptic curve $L$-isogeny graphs $\mathcal{G}(p,L)$ over $\Fpbar$ where $L = \{\ell_1,\dots, \ell_r\}$ is a collection of allowed prime isogeny degrees. This work is motivated by \cite{Delfs16} who studied such graphs over $\Fp$. In isogeny-based cryptographic protocols where multiple degrees are allowed, an $L$-isogeny graph is the structure underlying the security of the protocol. These graphs are $(\ell_1 + \cdots + \ell_r + r)$-regular, but no longer satisfy the Ramanujan property. Having more edges means these graphs have more \emph{cycles}. Cycles in $\ell$-isogeny graphs have been well-studied \cite{bank2019cycles,Eisentraeger20,fuselier2023,orvis2024}. A graph cycle starting at a vertex corresponding to the elliptic curve $E$ gives an endomorphism of $E$. Computing the endomorphism ring of a supersingular elliptic curve is an active area of research and a question which underlies almost all of isogeny-based cryptography. 

\subsection{Our contributions}

Our results concern two graph-theoretic structures in $L$-isogeny graphs: isogeny cycles (as defined in~\Cref{def:isogenycycle}) and cuts.

\begin{description}[wide, labelindent=0pt]

\item[Cycles.] We provide explicit counts of cycles in \(L\)-isogeny graphs in Section~\ref{sec:countingcycles}: one count using the theory of Brandt matrices in Section~\ref{ssec:brandt} and the other count using the theory of embeddings in quaternion algebras in Section~\ref{ssec:idealcounting}. These two sections count slightly different classes of cycles.

\item[Cuts.] We study edge expansion in the $L$-isogeny graph using methodology inspired by Fiedler cuts in Section~\ref{sec:graphcuts}. We provide a comparison with alternative algorithms for finding cuts with good edge expansion, and show heuristically that the greedy neighbour algorithm out-perform Fiedler's algorithm. 
\end{description}

We provide code in SageMath \cite{sage} at:
\url{https://github.com/jtcc2/cycles-and-cuts}.
This includes implementations of our counting arguments, clustering algorithms, figures, and examples.

\subsection{Related work}

While this work was in preparation, the authors were made aware of \cite{Kambe24}, which also studies cycles in $L$-isogeny graphs. The authors \cite{Kambe24} take a computational approach with the goal of computing endomorphism rings using cycles found in $L$-isogeny graphs, applying an approach similar to \cite{Eisentraeger20}. Their study of cycles restricts to a very special class, aligning with the work of \cite{Eisentraeger20}. Our work is a complementary extension: we consider the general class of isogeny cycles and provide explicit cycle counts.

\subsection{Acknowledgements}
This work was completed as part of the Banff International Research Station (BIRS) Isogeny Graphs in Cryptography Workshop in 2023. This workshop was held in-person at BIRS in Banff, Alberta, Canada as well as in-person in Bristol, UK. We would like to express our gratitude to the organisers, Victoria de Quehen, Chloe Martindale, and Christophe Petit.  This particular working group, led by Sarah Arpin, Giulio Codogni, and Travis Morrison, had participants in both Banff and Bristol. We received a lot of interest in this topic, and broke into smaller project groups. The authors are indebted especially to contributions by Giulio Codogni. The first and fifth authors are supported in part by faculty fellowships from the Commonwealth Cyber Initiative (CCI), an investment in the advancement of cyber R\&D, innovation, and workforce development. For more information about CCI, visit \url{www.cyberinitiative.org}. The authors further express their gratitude to their fellow group members for extremely helpful conversations, including: Kirsten Eisentr\"ager, Jun Bo Lau, William Mahaney, Travis Morrison, Eli Orvis, James Rickards, Maria Sabitova, Gabrielle Scullard, and Lukas Zobernig.  We extend our thanks to the anonymous reviewers for their careful reading of this manuscript and many suggestions to improve this work.

\section{Preliminaries}

\subsection{Supersingular elliptic curves}
We recall some standard facts about supersingular elliptic curves. For more detail, see \cite{ATAEC,VoightQuaternionAlgebras}.
An elliptic curve over a finite field $\mathbb{F}_q$ is supersingular if and only if the geometric endomorphism ring of $E$, $\End(E):=\End_{\overline{\mathbb{F}}_q}(E)$, is a maximal order in a quaternion algebra. The classical Deuring correspondence \cite{Deuring41} gives a categorical equivalence through which we can interpret isogenies of elliptic curves as left ideals of maximal orders in a quaternion algebra. To a given left ideal $I$ of $\End(E)$, we associate an isogeny $\varphi_I$ with kernel given by the scheme-theoretic intersection
\[\ker\varphi_I = \bigcap_{\alpha\in I}\ker\alpha.\]
When $I$ is a principal ideal, $\varphi_I$ is an endomorphism. If two left ideals $I$, $J$ of $\End(E)$ are in the same ideal class, the codomains of $\varphi_I$ and $\varphi_J$ are isomorphic. We almost have a group action: however, the set of left ideal classes of $\End(E)$ is not a group.

If $E$ is a supersingular elliptic curve defined over a prime field $\Fp$, the $\Fp$-part of its endomorphism ring, $\End_{\Fp}(E)$, is an imaginary quadratic order. In this case, we have a true action of the class group of ideals: the invertible 
ideals of the imaginary quadratic order $\End_{\Fp}(E)$ act freely on the set of supersingular elliptic curves with endomorphism ring isomorphic to $\End_{\Fp}(E)$. 

The $\Fpbar$-automorphism groups of supersingular elliptic curves are generally $\{[\pm1]\}$. For precisely two $\Fpbar$-isomorphism classes of elliptic curves are these automorphism groups larger: $j = 0, 1728$~\cite{Silverman}. Curves $E$ with $j(E) = 0$ or $1728$ are said to have \emph{extra automorphisms}.

\subsection{\texorpdfstring{$\ell$-isogeny graphs}{The l-isogeny graphs}}\label{ssec:l_isogenygraph}

\begin{definition}[{\texorpdfstring{$\ell$-isogeny graph}{l-isogeny graph}}]\label{def:lisogenygraph}
    For two primes \(p\) and \(\ell\), the $\ell$-isogeny graph, denoted $\mathcal{G}(p,\ell)$ is the graph with vertices given by isomorphism classes of supersingular elliptic curves over $\overline{\mathbb{F}}_p$, with edges given by degree-$\ell$ isogenies, up to post-composing with an automorphism.
\end{definition}

\emph{A priori}, an isogeny has a unique dual so we should be able to identify the edge corresponding to an isogeny with the edge corresponding to its dual and create an undirected graph. However, the definition of the edges in Definition~\ref{def:lisogenygraph} opens the door to some issues for vertices with automorphism groups larger than $[\pm1]$: if $\varphi_1:E_{1728}\to E$ and $\varphi_2:E_{1728}\to E$ have the same domain with $j$-invariant 1728 and the same codomain, it may happen that $\widehat{\varphi}_1 = [i]\circ\widehat{\varphi}_2$. 
In this case, there are two edges from the vertex $E_{1728}$ to  the vertex $E$, but only one edge from $E$ to $E_{1728}$. 
If we wish to create an undirected graph, we sacrifice the regularity of the graph. The issue can only arise adjacent to vertices with $j=0$ or $j=1728$, so given a fixed $\ell$, this problem can be avoided completely by choosing $p$ according to certain congruence conditions. 

If we associate to each edge an isogeny with explicit
rational equations, then a cycle in the \(\ell\)-isogeny graph can be composed into an endomorphism. However, if we only specify kernels and one of the vertices in the cycle is an isomorphism class with extra automorphisms, there is no canonical way of identifying this cycle with an endomorphism: composing with an extra automorphism may completely change the discriminant of the endomorphism obtained (where the discriminant of an endomorphism is defined as in \cite[Section 2C]{Eisentraeger20}, via the Deuring correspondence \cite{Deuring41}). There are finitely many extra automorphisms, so the set of possibilities is finite - the issue is identifying them. To handle this issue, the notion of \emph{arbitrary assignment} was developed in \cite{Arpin2022OrientationsAC}.

\begin{definition}[Arbitrary assignment, {\cite[Definition 3.3]{Arpin2022OrientationsAC}}]\label{def:arb_assign}
Given an $\ell$-isogeny graph $\mathcal{G}(p,\ell)$, an arbitrary assignment is a choice (up to sign) of equivalence class representative $\pm\varphi$ for every edge $[\varphi]$. 
\end{definition}

For edges yielding isogenies where the codomain $j$-invariants are not equal to 0 or 1728, an arbitrary assignment is automatic and does not involve any choice. When $p\equiv1\pmod{12}$, such an assignment is not necessary as $[-1]$ commutes with every isogeny, so this can always be a post-composition, which is already an equivalence we place on edges.

\subsection{\texorpdfstring{$L$-isogeny graphs}{L-isogeny graphs}}\label{ssec:bigLisogenygraphs}

\begin{notation}\label{notation}
    Let $p,\ell_1, \dots, \ell_r$ be distinct primes.
    Let $L:=\{\ell_1,\dots,\ell_r\}$. For convenience and without loss of generality, suppose $\ell_1<\ell_2<\cdots<\ell_{r}$. In what follows, $e_i\in\mathbb{Z}_{\geq0}$ for all $i=1,\dots,r$.
\end{notation}

In \cite{wissam2022loops}, the author studies two isogeny graphs $\mathcal{G}(p,\ell_1), \mathcal{G}(p,\ell_2)$  simultaneously. We now consider a generalisation to an arbitrary number of isogeny graphs $\mathcal{G}(p,\ell_i)$, for $i=1,\dots, r$.

\begin{definition}[$L$-isogeny graph]\label{def:Lisogenygraph}
    The $L$-isogeny graph, denoted $\mathcal{G}(p,L)$ is the graph with vertices given by isomorphism classes of supersingular elliptic curves over $\Bar{\mathbb{F}}_p$, with edges given by degree $\ell_i$-isogenies (up to post-composing with an automorphism) for $i = 1,\ldots,r$. 
\end{definition}

As defined, $L$-isogeny graphs are directed graphs. As in the $\ell$-isogeny case described in Section~\ref{ssec:l_isogenygraph}, they are undirected for appropriate congruence conditions on $p$.
See Figure~\ref{fig:examplecounting_2} for the $\{2,3\}$-isogeny graph of supersingular elliptic curves over $\overline{\mathbb{F}}_{61}$.

\subsection{Brandt matrices} \label{Brandtpremiliminaires}
Let $\mathcal{B}_p$ denote the (unique up to isomorphism) quaternion algebra over $\mathbb{Q}$ ramified at $p$ and $\infty$.  Two left ideals $I, J$ of an order in $\mathcal{B}_p$ are \textit{equivalent} if there exists $\beta \in \mathcal{B}_p^{\times}$ such that $J=I\beta$. 
	The class set $cl(\mathfrak{O}) = \{ I_1, I_2, \dots,I_n\}$, for a maximal order $\mathfrak{O}$ in $\mathcal{B}_p$, is the set of representatives of equivalence classes of left $\mathfrak{O}$-ideals. We let $I_1=\mathfrak{O}$ by convention.
	The class group is finite and its cardinality $n$ is called the \textit{class number} of $\mathcal{B}_p$ (it is the same for every maximal order $\mathfrak{O}$ in $\mathcal{B}_p$). 
Let $O_R(I)$ denote the right order of the ideal $I$. The set $\Gamma_i := O_R(I_i)^{\times}/\mathbb{Z}^{\times}$ is finite, as it is a discrete subgroup of the compact Lie group $(\mathcal{B}_p \otimes \mathbb{R})^{\times}/\mathbb{R}^{\times} \cong \text{SO}_3(\mathbb{R})$. Let $w_i$ be its cardinality. See \cite{VoightQuaternionAlgebras}.

	Following \cite{gross1987heights}, we now introduce theta series and Brandt matrices. 
	For an ideal \(I_i\), define the inverse ideal \(I_i^{-1}\) as $I_{i}^{-1} := \{ \beta \in \mathcal{B}_p : I_{i}\beta I_{i} \subseteq I_{i} \}$, and let $M_{ij}:= I_j^{-1}I_i = \{ \sum_{k=1}^N a_k b_k : N \in \mathbb{Z}_{\geq 1},  a_k \in  I_j^{-1}, b_k \in I_i \}$. 
	We define the reduced norm $\text{Nm}(M_{ij})$ of $M_{ij}$ to be the unique positive rational number such that all quotients $\text{Nm}(a)/\text{Nm}(M_{ij})$, for $a \in M_{ij}$, form a coprime set of integers. 
    We now define the following \textit{theta series} $\theta_{ij}$:
		\[\theta_{ij}(\tau) := \frac{1}{2 w_j} \sum_{a \in M_{ij}} 
		q^{\frac{\text{Nm}(a)}{\text{Nm}(M_{ij})}}
						= \sum_{m \ge 0} B_{ij}(m)q^{m}, \]	
	where $q:= e^{2 \pi i \tau}$ and $\tau \in \{z \in \mathbb{C} : \text{Im}(z) >0\}$.
    Finally, we define the \textit{Brandt matrix} $B(m)$ to be $B(m):= \begin{bmatrix} B_{ij}(m) \end{bmatrix}_{1 \le i,j \le n}$, for $m \in \mathbb{Z}_{\geq 1}$. Note that $B(1)$ is the identity matrix. 
	We have the following properties for Brandt matrices. 
    \begin{proposition}[{\cite[Proposition 2.7]{gross1987heights}}] \label{BrandtProps}\
	\begin{enumerate}[label=(\roman*)]
    
		\item
		If $m \ge 1$, then $B_{ij}(m)\in \mathbb{Z}_{\geq0}$, and for all $1 \le i \le n$,
			\[\sum_{j=1}^n B_{ij}(m) = \sum_{\substack{d | m \\\gcd(d, p)=1}} d. \]
		\item
		If $m$ and $m'$ are coprime, then $B(mm')= B(m)B(m')$. 
		\item
		If $\ell \neq p$ is a prime, then 
            \begin{equation}
            \label{recursiveBrandtitem}
            B(\ell^{k}) = B(\ell^{k-1}) B(\ell) - \ell B(\ell^{k-2}), 
            \end{equation}    
            for all $k \ge 2$. 
		\item
		$w_j B_{ij}(m) = w_i B_{ji}(m)$ for all $m$ and for all $1 \leq i,j \leq n$. 
	\end{enumerate} 
\end{proposition}
The main reason for considering Brandt matrices in our study of isogeny graphs is the fact that $B_{ij}(m)$ is the number of equivalence classes of isogenies of degree $m$ between the elliptic curves corresponding to $I_i$ and $I_j$ (cf.~\cite[Proposition 2.3]{gross1987heights}). Therefore, the Brandt matrix $B(\ell)$ is the adjacency matrix of the supersingular isogeny graph $\mathcal{G}(p,\ell)$. Hence, in order to study cycles of length $r$ in the graph $\mathcal{G}(p,\ell)$, one can alternatively look at the entries $B(\ell^r)$ to count those closed walks based at $I_i$, or at $\Tr(B(\ell^r))$ to count them all (counting all starting points). See \cite{wissam2022loops, GKPV_Collisions} for more details on this approach, in the case of two primes. 

    Now, computing the entries of $B(\ell^r)$ can be done using the recursion relation of~\Cref{recursiveBrandtitem}, but this can be very inefficient for larger powers of $\ell$ or large primes $p$. However, there is a way of computing certain traces of Brandt matrices, by re-expressing them in terms of modified Hurwitz class numbers, which we define below. 

	Given an order $\mathcal{O}$ in an imaginary quadratic field, let $d$ 
	be its (negative) discriminant, $h_\mathcal{O}$ the size of the class group $cl(\mathcal{O})$, and $u_\mathcal{O} := | \mathcal{O}^{\times}/\mathbb{Z}^{\times}|$.  
	To emphasize the discriminant $d$ of an order, we will denote the order $\mathcal{O}_d$. Fix $D>0$ for which there exists an order $\mathcal{O}_{-D}$ of discriminant $-D$. The \emph{Hurwitz Class Number} $H(D)$ is 
		\begin{equation*}
			H(D) = \sum_{\substack{d \in \mathbb{Z},\mathfrak{f} >0  \\ d \cdot \mathfrak{f}^{2} =-D}} \frac{h_{\mathcal{O}_d}}{u_{\mathcal{O}_d}}.
		\end{equation*} 
	and the \emph{modified Hurwitz Class Number} $H_p(D)$, for a prime $p$, is 
			\begin{equation}
			\label{HCNDef} 
			H_p(D):= \left\{
						\begin{array}{ll}
							0  & \mbox{if } p \text{ splits in } \mathcal{O}_{-D};  \\[3pt]
							H(D)  & \mbox{if } p \text{ is inert in } \mathcal{O}_{-D}; \\[3pt] 							\frac{1}{2}H(D)  & \arraycolsep0pt\begin{array}{l}\mbox{if } p \text{ is ramified in } \mathcal{O}_{-D} \\ 
									 \text{but does not divide the conductor of } \mathcal{O}_{-D};\end{array} \\[
9pt]							H(\frac{D}{p^{2}})  & \mbox{if } p \text{ divides the conductor of } \mathcal{O}_{-D} 
						\end{array}
					\right. 	
			\end{equation} 
	For $D=0$, we set $H(0) = -1/12$ and $H_p(0) := \frac{p-1}{24}$. 
Note that $H_p(D) \le H(D)$. 
We conclude this section with the following two results on Hurwitz class numbers. The first connects Hurwitz class numbers to the trace of Brandt matrices. The second gives us a way to compute certain sums of Hurwitz class numbers. 

\begin{theorem}[{\cite[Prop. 1.9]{gross1987heights}}]\label{GrossBrandtToHurwitz}

	For all integers $m \ge 0$, 
			\begin{equation*} 
				\Tr(B(m)) = \sum_{\substack{s \in \mathbb{Z} \\ s^2 \le 4 m}} H_p (4 m - s^2).
			\end{equation*} 
\end{theorem}
\begin{theorem}[{\cite[Section 7]{hurwitz1885ueber}}] \label{GrossHurwitzSum}
	For all integers $m \ge 1$, 
	\begin{equation*} 
		\sum_{\substack{s \in \mathbb{Z} \\ s^2 \le 4 m}} H(4m - s^{2} ) = 2 \sum_{\substack{d | m \\ d > 0}} d - \sum_{\substack{d | m \\ d> 0}} \min\{d,m/d\}. 
	\end{equation*} 
\end{theorem}

\subsection{Matrix representations of the graph}

Let $G = (V,E)$ be a graph. Fix an ordering of the vertices $V = (v_1,\dots,v_n)$. The \emph{adjacency matrix} of $G$ with respect to the ordering $(v_1,\dots,v_n)$ is the $n\times n$ matrix $A$ whose entries $a_{ij}$ are the number of edges from $v_i$ to $v_j$. If the graph is undirected, the adjacency matrix is symmetric. If the graph is directed, this may or may not be the case. 

For an undirected graph, the degree of a vertex is the number of edges incident to that vertex. If the graph is directed, we define the in- and out-degree of a vertex to be the number of in- and out-edges, respectively. Given an ordering of the vertices $V = (v_1,\dots,v_n)$, the (in- or out-)degree matrix is the diagonal matrix $D$ whose nonzero entries $d_{ii}$ are the (in- or out-)degrees of the vertices $v_i$.

The Laplacian of an undirected graph is the matrix $L:= D-A$. If $G$ is directed, we can define the in- or out-Laplacian where $D$ is the in- or out-degree matrix, respectively.

\subsection{Edge Expansion}\label{prelim_fielder}

A \emph{cut} of a graph \(G = (V,E)\) is an ordered partition of the vertices \(V = V_1 \sqcup V_2\) with \(V_1, V_2 \neq \varnothing\). The sets \(V_1\) and \(V_2\) are called the \emph{sides} of the cut.
The directed edges \( e = (u, v) \in E\) which have \(u \in V_1\) and \(v \in V_2\) (or vice versa),
and undirected edges with one endpoint in each of \(V_1\) and \(V_2\),
are said to \emph{cross} the cut,
and the set of such edges is called the \emph{cut set} of the cut, denoted \(E(V_1, V_2)\). The cuts which are usually of most interest in graph theory are those which are relatively balanced (i.e., in which \(|V_1| \approx |V_2| \)) and relatively \emph{sparse} (i.e., they have relatively few edges which cross the cut). These desirable properties are incorporated in the \emph{edge expansion}, defined in Definition~\ref{def:EdgeExpansion}.

\begin{definition}[Edge expansion]
\label{def:EdgeExpansion}
Let \(G = (V,E)\) be a graph. The \emph{edge expansion} of a cut \(V = C \sqcup (V \setminus C)\) is
\[
h_G(C) := \frac{|E(C, V \setminus C)|}{\min\{ \Vol(C), \Vol(V \setminus C) \}},
\]
where $\Vol(C)$ is the sum of the vertex degrees in $C$, which means for $d$-regular graphs, $\Vol(C) = d \cdot |C|$.
The edge expansion of \(G\) is \(h(G) := \min\limits_{\varnothing \subsetneq T \subsetneq V} h_G(T)\).
\end{definition}

\section{\texorpdfstring{Cycles in the $L$-isogeny graph}{Cycles in the L-isogeny graph}}\label{sec:cycles} 
Cycles in isogeny graphs are of interest for two main reasons:
\begin{enumerate} 
    \item They correspond to collisions of walks in the isogeny graph \cite{CGL09};
    \item They correspond to endomorphisms of elliptic curves and can be used to compute endomorphism rings \cite{bank2019cycles,Eisentraeger20,fuselier2023,Kambe24}.
\end{enumerate}

In this section and the following section, we restrict ourselves to the case $p\equiv1\pmod{12}$, which allows us to treat $\mathcal{G}(p,\ell)$ as an undirected graph by identifying each isogeny with its dual, 
thus avoiding counting issues arising from extra automorphisms. Moreover, we assume that $L$ never contains $p$.

\subsection{Setup} 

\begin{definition}[Isogeny cycle]\label{def:isogenycycle}
    An $\{\ell_1^{e_1},\ldots,\ell_r^{e_r}\}$-isogeny cycle\footnote{We note that this is different from a cycle in the typical graph-theoretic sense, which is a closed \emph{path}---i.e., a closed walk with no repeated vertices.} is a closed walk in $\mathcal{G}(p,\{\ell_1,\dots,\ell_r\})$ 
    where the set of edges in the walk consists of $e_i$ degree-$\ell_i$ isogenies for $i = 1,\dots,r$, with $e_i>0$.
    We denote the set of all $\{\ell_1^{e_1},\ldots,\ell_r^{e_r}\}$-isogeny cycles in $\mathcal{G}(p,\{\ell_1,\dots,\ell_r\})$ by $\operatorname{Cycles}_p(\ell_1^{e_1},\ldots,\ell_r^{e_r})$. 
\end{definition}

\begin{remark}
    It is possible to discuss $\{\ell_i^{e_i}\}_{i=1}^r$-isogeny cycles in an $L$-isogeny graph $\mathcal{G}(p,L)$ where $\{\ell_1,\dots,\ell_r\}\subsetneq L$, but since the cycle does not contain edges of degree $\ell$ for any $\ell\in L\setminus \{\ell_1,\dots,\ell_r\}$, we may simply consider this isogeny cycle as belonging to the graph $\mathcal{G}(p,\{\ell_1,\dots,\ell_r\})$. Going forward, when we discuss $\{\ell_1^{e_1},\dots,\ell_r^{e_r}\}$-isogeny cycles, we automatically define $L:=\{\ell_1,\dots,\ell_r\}$.
\end{remark}

Given a walk (resp. closed walk) in $\mathcal{G}(p, L)$ as a sequence of edges, we may successively compose compatible isogeny representatives corresponding to these edges to obtain an isogeny (resp. endomorphism) of elliptic curves. Each vertex represents an isomorphism class of elliptic curves, so the representative of each edge source is successively chosen in a compatible way.\footnote{This is possible as no additional automorphisms exist, so pre- and  post-composition of an isogeny with automorphisms will always correspond to the same edge.} The reverse process involves decomposing an isogeny into a sequence of prime degree isogenies.

\begin{definition}\label{def:decomposition}
    Let $\varphi : E \to E'$ be a separable isogeny of degree divisible only by primes contained in $L$. A (prime) decomposition of $\varphi$ is a factorisation of $\varphi$ into a sequence of (prime degree) isogenies
    \[
        \varphi = \varphi_k \circ \varphi_{k-1} \circ \cdots\circ \varphi_2 \circ \varphi_1 =: (\varphi_i)_{i=1}^k. 
    \]
    Define $\mathcal{D}(\varphi)$ to be the set of prime decompositions of $\varphi$.
\end{definition}

A prime decomposition yields a graphical representation of an isogeny as a walk in $\mathcal{G}(p, L)$, since each $\varphi_i$ corresponds to an edge. In the context of the discussion above, we treat prime decompositions as equivalent if they yield the same graph-theoretic walk. 
Since we are assuming
$p\equiv 1 \mod 12$, every elliptic curve has $\Aut(E) = \{\pm 1\}$, so two prime decompositions are equivalent when they differ by replacing $\varphi_i$ with $-\varphi_i$ for any subset of indices $i$. 
Even up to this equivalence, a prime decomposition need not be unique.

\begin{lemma}\label{lem:count_decomp}
    Let $\varphi : E \to E'$ be an isogeny with kernel $G$ such that $\deg\varphi$ is divisible only by primes contained in $L$. Let $(P, \triangleleft)$ be the poset of subgroups of $G$, ordered by inclusion. There is a 1-1 correspondence between the prime decompositions of $\varphi$ (elements of $\mathcal{D}(\varphi))$ and the maximal chains in $(P, \triangleleft)$.
\end{lemma}
\begin{proof}
    Given a prime decomposition $\varphi = (\varphi_i)_{i=1}^k$, define the sequence of groups
    \[
        G_0 = \{O_E\}, G_{i} = \ker(\varphi_i \circ \varphi_{i-1}\circ \cdots \circ \varphi_1).
    \]
    This sequence forms a chain in $P$ as each $\varphi_i$ is a group homomorphism. Suppose there exists $H \in P$ such that $G_{i-1} \triangleleft H \triangleleft G_i$. Comparing indices, we have $\deg\varphi_i = |G_i/G_{i-1}| = |G_i/H||H/G_{i-1}|$, and as $\deg\varphi_i$ is prime, either $H = G_i$ or $H=G_{i-1}$. As $G_k = G$, we conclude that this chain is maximal.
    \par
    For the reverse direction, let $\{O_E\} = G_0 \triangleleft G_1 \triangleleft \cdots \triangleleft G_k = G$ be a maximal chain in $P$. For each $1\le i \le k$, there exists a unique curve $E_i$ (up to isomorphism) and isogeny $\psi_i : E \to E_i$ (up to sign) with $\ker \psi_i = G_i$ (\cite[Chapter III, Proposition 4.11]{Silverman}). Define $\varphi_1 = \psi_1$, and note that $\psi_k = \varphi$. For $i>1$, define $\varphi_i$ to be the unique isogeny satisfying $\psi_i = \varphi_i\circ \psi_{i-1}$ (\cite[Chapter III, Corollary 4.11]{Silverman}). This procedure yields a decomposition $\varphi = \psi_k = \varphi_k\circ \varphi_{k-1}\circ \cdots \circ \varphi_1$ where $\deg\varphi_i = |G_i/G_{i-1}|$. Each of these indices is prime as the $G_i$'s form a maximal chain. Finally, these correspondences are inverse of each other up to equivalence of decompositions.
\end{proof}

As isogenies have finite kernels, $\mathcal{D}(\varphi)$ is always a finite set. 
We can compute $|D(\varphi)|$ from the prime power decomposition of $\deg\varphi$ together with knowledge of the largest integer $n$ such that $\varphi(E[n]) = O_{E'}$. 
As a special case, suppose the isogeny $\varphi$ is primitive (\cite[Definition 4.1]{bank2019cycles}; equivalently, has cyclic kernel); then, we have $n=1$ and writing  $\deg\varphi = \prod_{i=1}^r\ell_i^{e_i}$ 
we have 
\begin{equation}\label{eq:Dvarphi}
    |\mathcal{D}(\varphi)| = \binom{\sum_{i=1}^r e_i}{e_1, e_2, \hdots, e_r} = \frac{(\sum_{i=1}^re_i)!}{\prod_{i=1}^re_i!}
\end{equation}
using a simple combinatorial argument~\cite[Section 5.3, Theorem 1]{tucker2007applied}. In particular, if the number of distinct primes dividing $\deg\varphi$ is greater than 1, then $|D(\varphi)|>1$: such isogenies do not have a unique prime decomposition. 

\subsection{Backtracking} 
A walk in the $\ell$-isogeny graph composes to an isogeny with non-cyclic kernel if and only if the walk contains backtracking~\cite[Definition 4.3]{bank2019cycles}. In particular, there is no canonical choice of walk in the $\ell$-isogeny graph corresponding to the multiplication-by-$\ell$ isogeny of a particular starting vertex. 

\begin{definition}[Backtracking]\label{def:backtracking}
    An isogeny cycle given by the composition  $\varphi_k\circ\varphi_{k-1}\circ\cdots\circ\varphi_1$ has \emph{backtracking} if $\widehat{\varphi}_i\in \Aut(E_i)\varphi_{i+1}$ for some $i\in\{1,2,\dots,k-1\}$ where $\varphi_i:E_i\to E_{i+1}$ and $\varphi_{i+1}:E_{i+1}\to E_{i+2}$.
\end{definition}

In $\mathcal{G}(p,\ell)$, a cycle has backtracking if and only if the endomorphism obtained by composing (compatible isogeny representatives of) edges in the cycle has non-cyclic kernel \cite[Proposition 4.5]{bank2019cycles}. 
With $L$-isogeny graphs, the kernel of an isogeny which is the composition of a non-backtracking walk can still be non-cyclic, as demonstrated in~\Cref{ex:IsogenyDiamond}. 

\begin{example}[A Non-Cyclic Kernel Arising from an Isogeny Diamond] \label{ex:IsogenyDiamond}
Let $E_1$ be an elliptic curve, and suppose $\varphi\colon E_1 \to E_2,\psi \colon E_1 \to E_3$ are two isogenies of distinct prime degrees, each with domain $E_1$.
Define isogenies $\varphi'$ from $E_3$ and $\psi'$ from $E_2$ with $\ker\varphi' = \psi(\ker\varphi)$ and $\ker\psi' = \varphi(\ker\psi)$. The diagram in Figure~\ref{fig:isogenydiamond} commutes (up to post-composition
by an automorphism) \cite[Corollary III.4.11]{Silverman}.

\begin{figure}[htbp]
    \centering
    \begin{tikzpicture}
    \node[] (E1) at (0,0) {$E_1$};
    \node[] (E2) at (0,-2) {$E_2$};
    \node[] (E3) at (2,0) {$E_3$};
    \node[] (E4) at (2,-2) {$E_4$};

    \draw[->] (E1) to node[left] {$\varphi$} (E2);
    \draw[->] (E1) to node[above] {$\psi$} (E3);
    \draw[->] (E3) to node[right] {$\varphi'$} (E4);
    \draw[->] (E2) to node[below] {$\psi'$} (E4);
    \end{tikzpicture}
    \caption{An isogeny diamond.}
    \label{fig:isogenydiamond}
\end{figure}
The isogeny diamond gives a $\{(\deg\varphi)^2,(\deg\psi)^2\}$-isogeny cycle in the $\{\deg\psi,\deg\varphi\}$-isogeny graph, namely $\widehat{\varphi}\circ\widehat{\psi'}\circ\varphi'\circ\psi=:\alpha\in\End(E_1)$. We see that $\alpha$ and $[\deg\psi]\circ[\deg\varphi]$ are equivalent by computing $\alpha(E[\deg\psi])=\alpha(E[\deg\varphi])=O_E$.
\end{example}

Indeed, the isogeny cycle $\widehat{\varphi}\circ\widehat{\psi'}\circ\varphi'\circ\psi \in\End(E_1)$ from Example~\ref{ex:IsogenyDiamond} does not have backtracking. By refactoring the cycle as $[\deg\psi]\circ[\deg\varphi]$ so that isogenies of the same degree are grouped, we see that it has non-cyclic kernel.

\begin{lemma}[Isogeny swapping]\label{lem:isog_swap}
    Suppose we are given an isogeny as the composition $\varphi'\circ\psi$ of two isogenies of coprime degrees. Define an isogeny $\psi'$ with $\ker\psi':= \varphi(\ker\psi)$ and an isogeny $\varphi$ with $\ker\varphi := \widehat{\psi}(\ker\varphi')$. Then,
    \[\varphi'\circ\psi = \alpha\circ\psi'\circ\varphi,\]
    where $\alpha$ is an automorphism of the codomain of $\varphi'$.
\end{lemma}

\begin{proof}
    The result follows directly from the definitions of $\varphi$ and $\psi'$: the resulting diagram is an isogeny diamond (see~\Cref{fig:isogenydiamond}). 
\end{proof}

\subsection{Canonical Decomposition}
To handle the issue of multiple cycles composing to the same endomorphism, we provide a canonical means of refactoring a given isogeny cycle: we refactor the isogeny using Lemma~\ref{lem:isog_swap}, grouping the isogenies by increasing prime degree ($\ell_1<\ell_2<\cdots<\ell_r$, as in Notation~\ref{notation}). Consecutive isogenies of the same prime degree cannot be reordered; however they may compose to a scalar multiplication map, which may be reordered arbitrarily, so we obtain a canonical decomposition which pushes scalar multiplication maps together.

\begin{definition}[Canonical decomposition]\label{def:canonical}

 Let $\ell_1<\ell_2<\cdots<\ell_r$ be distinct primes. Given a fixed $\{\ell_1^{e_1},\ldots,\ell_r^{e_r}\}$-isogeny cycle, let $\alpha$ denote the endomorphism given by the composition of the isogenies in this isogeny cycle.

 There is a largest positive integer $n$ such that $\alpha(E[n]) = O_E$.
Write $\frac{\deg(\alpha)}{n^2} = \ell_{i_1}^{f_1} \cdots \ell_{i_m}^{f_m}$ where $i_1, \dots, i_m$ is a subsequence of $1,\dots,r$ such that each $f_i$ is non-zero.
 
The \emph{canonical decomposition} of $\alpha$ is the decomposition: 
\begin{equation}\label{eq:canon_decomp}
\alpha = [n]\circ(\varphi_{m, f_m}\circ\cdots\circ\varphi_{m, 1})\circ\cdots\circ(\varphi_{2, f_2}\circ\cdots\circ\varphi_{2, 1})\circ(\varphi_{1, f_1}\circ\cdots\circ\varphi_{1, 1}),
\end{equation}
where $\deg\varphi_{j, k} = \ell_{i_j}$ for all \(k\). 

We define the \emph{$L$-isogeny graph walk associated to the canonical decomposition} of $\alpha$ to be the walk given by the list of edges: 
\[\varphi_{m, f_m},\cdots,\varphi_{m, 1},\cdots,\varphi_{2, f_2},\cdots,\varphi_{2, 1},\varphi_{1, f_1},\cdots,\varphi_{1, 1}.\]
\end{definition}

\begin{remark}\label{rmk:[n]notwalk}
    We do note associate any isogeny walk in the \(L\)-isogeny graph to the $[n]$ that appears in \Cref{eq:canon_decomp}, since there are many walks in the graph whose associated isogenies compose to \([n]\) and there is no canonical way to choose one. For example, there are three walks in the 2-isogeny graph corresponding to $[2]$ at a particular starting vertex, but $[2]$ itself is not an edge in the graph.  
    
    However, the endomorphism $(\varphi_{m, f_m}\circ\cdots\circ\varphi_{m, 1})\circ\cdots\circ(\varphi_{2, f_2}\circ\cdots\circ\varphi_{2, 1})\circ(\varphi_{1, f_1}\circ\cdots\circ\varphi_{1, 1})$ corresponds to a uniquely-determined $\{\ell_{i_j}^{f_j}\}$-isogeny cycle.
\end{remark}

This motivates the following algorithm for finding the canonical decomposition of an isogeny cycle.

\begin{savenotes}
\begin{algorithm}
    \caption{Finding a canonical decomposition of an endomorphism}\label{alg:canonicaldecomp}
    \Input {An endomorphism $\alpha$ of a supersingular elliptic curve $E$ given as the composition of a list of (known) prime-degree isogenies: $(\varphi_{1},\dots,\varphi_{N})$}
        
    \Output {An ordered list $(\psi_1,\dots,\psi_M)$ of isogenies of non-decreasing prime degree, and an integer $n$ such that the composition $[n]_E\circ\psi_M\circ\cdots\circ\psi_1$ is equal to $\alpha$ up to sign. }

    Apply BubbleSort\footnote{Lemma~\ref{lem:isog_swap} allows us to swap \emph{adjacent} isogenies of different degrees in an isogeny cycle. Iteratively applying this idea, we can sort the isogenies in increasing degree by applying any sorting algorithm that only makes adjacent swaps. BubbleSort is one such algorithm~\cite{friend1956bubble}.} to the list $\Phi:=(\varphi_{1},\dots,\varphi_{N})$ to sort by increasing degree, using the isogeny swapping procedure in Lemma~\ref{lem:isog_swap}\;
    \label{line:BubbleSort}

    $n_0 \gets 1$\;\label{line:Initialize_n0}
    $i \gets 1$\;\label{line:Iniitialize_i}
    \While{$i \leq \mathsf{len}(\Phi)$}{
        \label{line:StartWhile}
        Let \(E_i\) be the domain of \(\Phi[i]\)\;\label{line:Domain}
        \If{$\Phi[i+1]\circ\Phi[i] \in \Aut(E_i)[\deg\Phi[i]]$}{
        \label{line:StartIf}
            $n_0 \gets n_0 \cdot \deg\Phi[i]$\;\label{line:Update_n0}
            Remove $\Phi[i],\Phi[i+1]$ from $\Phi$\;\label{line:RemoveDuals}
            $i \gets \textsf{max}(1, i - 1)$\;\label{line:Reset_i}
        }
        \Else{
            $i \gets i + 1$\;\label{line:Increment_i}
        }
    }

    \Return $\Phi, n_0$.\;
\end{algorithm}
\end{savenotes}

\begin{proposition}\label{lem:alg_correct}
    Algorithm~\ref{alg:canonicaldecomp} computes a canonical decomposition of an isogeny cycle. 
\end{proposition}

\begin{proof}
Let $\alpha$ be the endomorphism obtained from an $\{\ell_1^{e_1},\ldots,\ell_r^{e_r}\}$-isogeny cycle represented as a prime decomposition $(\varphi_1, \dots, \varphi_N)$.
Let $(\tilde{\Phi}, \tilde{n})$ be the output of Algorithm \ref{alg:canonicaldecomp} on input $\alpha$,  where $\tilde{\Phi} = (\psi_1, \dots \psi_M)$.
        Let $\beta = \psi_M \circ \cdots \circ \psi_1$, and let $n$ be the largest integer such that $E[n] \subseteq \ker\alpha$. The list $\Phi$ obtained after \cref{line:BubbleSort} gives a prime decomposition of $\alpha$, with the isogenies ordered by increasing degree. This property is preserved by the loop, as well as the invariant $\alpha = [n_0]_E \circ \Phi[\mathsf{len}(\Phi)] \circ \dots \circ\Phi[1]$. Hence the identity $\alpha = [\tilde{n}] \circ \beta$ is satisfied. The internal loop variable $n_0$ is non-decreasing, hence we have $\tilde{n} \le n$. To show equality, note that if $\beta$ factors through the scalar map $[\ell]$ for some $\ell\in L$, then $\psi_b\circ\dots\circ\psi_a$ must factor through $[\ell]$ where $\psi_a, \dots \psi_b$ are the elements of $\tilde{\Phi}$ of degree $\ell$. In turn, an $\ell^\star$-isogeny factors through $[\ell]$ if and only if it is backtracking, i.e., there exists an index $c$ with $a \le c < b$ such that $\psi_{c+1} = \pm\widehat{\psi_c}$, because $p\equiv 1\pmod{12}$. This condition is prohibited in the output by the conditional on \cref{line:StartIf}. 

        To see that the canonical decomposition is unique, note that the degree determines the ordering of the prime-power parts of the walk. The decomposition of a prime-power degree isogeny is unique up to composition with automorphisms at the vertices. Thus, if the vertices in the walk have only $[\pm 1]$ automorphisms, the canonical decomposition is unique up to $[\pm 1]$.
\end{proof}

\begin{corollary}\label{prop:cyclickernel}
    The $L$-isogeny graph walk corresponding to the canonical decomposition of an isogeny cycle contains no backtracking and the corresponding endomorphism has cyclic kernel.
\end{corollary}
\begin{proof}
    The scalar portion of the canonical decomposition is not part of the walk of the graph, see Remark~\ref{rmk:[n]notwalk}.
    Since the canonical decomposition of an isogeny cycle groups together isogenies of the same prime degree, it contains backtracking if and only if it contains a walk of degree-$\ell_i$ isogenies which contains backtracking \cite[Prop. 4.5]{bank2019cycles}.
\end{proof}

Every isogeny can be written as $[n]\circ\phi$, where $\phi$ has cyclic kernel. If $[n] = [1]$, we will call the isogeny cycle \emph{principal}.

\begin{definition}[Principal isogeny cycles]\label{def:principal}
    An $\{\ell_1^{e_1},\dots, \ell_r^{e_r}\}$-isogeny cycle is principal if its canonical decomposition has scalar factor $[1]$. 
    Let $\operatorname{Cycles}_p(\ell_1^{e_1},\ldots,\ell_r^{e_r})^{\mathrm{prin}}$ denote the subset of principal cycles. 
\end{definition} 
The endomorphism obtained from an isogeny cycle in $\operatorname{Cycles}_p(\ell_1^{e_1},\ldots,\ell_r^{e_r})^{\mathrm{prin}}$ has cyclic kernel, by Corollary ~\ref{prop:cyclickernel}.

\begin{figure}
    \centering
    \scalebox{.7}{
    \begin{tikzpicture}[xscale=1, yscale=1.4]
\node[circle,minimum size=.8cm,draw,line width=.4mm] (E1) at (1,3) {$j_1$};
\node[circle,minimum size=.8cm,draw,line width=.4mm] (E2) at (2,1) {$j_5$};
\node[circle,minimum size=.8cm,draw,line width=.4mm] (E3) at (3.5,4) {$j_2$};
\node[circle,minimum size=.8cm,draw,line width=.4mm] (E4) at (5,1) {$j_4$};
\node[circle,minimum size=.8cm,draw,line width=.4mm] (E5) at (6,3) {$j_3$};

\draw[<-, line width=.6mm, dashed] (E1) -- (E2) node [midway, fill=white] {$\varphi_{3,3}$};
\draw[<-, line width=.6mm] (E3) -- (E1) node [midway, fill=white] {$\varphi_{2,1}$};
\draw[->, line width=.6mm,dashed] (E3) -- (E5) node [midway, fill=white] {$\varphi_{3,1}$};
\draw[<-, line width=.6mm] (E4) -- (E5) node [midway, fill=white] {$\varphi_{2,2}$};
\draw[->, line width=.6mm,dashed] (E4) -- (E2) node [midway, fill=white] {$\varphi_{3,2}$};
\end{tikzpicture}
}
\caption{A $\{2^2,3^3\}$-isogeny cycle in the $\{2,3\}$-isogeny graph. The (solid line) isogenies $\varphi_{2,1}$ and $\varphi_{2,2}$ are degree-2 and the (dashed line) isogenies $\varphi_{3,1},\varphi_{3,2},\varphi_{3,3}$ are degree-3. This $\{2^2,3^3\}$-isogeny cycle can be specified by the tuple $(\varphi_{3,3},\varphi_{3,2},\varphi_{2,2},\varphi_{3,1},\varphi_{2,1})$, starting at the vertex $j_1$.}
\label{fig:L-isog-cycle}
\end{figure}

\begin{example}[Isogeny cycle]\label{ex:isocycle}
    Consider the abstract $\{2^2,3^3\}$-isogeny cycle beginning at vertex $j_1$ depicted in Figure~\ref{fig:L-isog-cycle}. Let $\theta\in\End(E_{j_1})$ denote the endomorphism obtained by composing (compatible isogeny representatives of) the edges of this cycle. The endomorphism $\theta$ can be written as a composition of prime degree isogenies in a total of $\binom{5}{2} = 10$ ways, each with two degree-2 isogenies and three degree-3 isogenies. The canonical decomposition of this endomorphism will be of the form $\psi_{3,3}\circ\psi_{3,2}\circ\psi_{3,1}\circ\psi_{2,2}\circ\psi_{2,1}$, where the $\psi_{i_j}$ isogenies are chosen so that the kernel of the composition remains unchanged. In particular, 
    \[\ker\psi_{2,2} = \ker(\varphi_{3,3}\circ\varphi_{3,2}\circ\varphi_{2,2}\circ\varphi_{3,1})\cap E(j_2)[2],\]
    where $E(j_2)$ is the codomain of $\varphi_{2,1}$. 
\end{example}

\begin{example}[Isogeny cycle which is not principal]\label{ex:npisocycle}
Here, we give an explicit example of an isogeny cycle in non-canonical form which does not contain backtracking, but whose canonical form does contain backtracking (it has a scalar factor $[2]$). See Figure~\ref{fig:decomp}.

Let $p = 2689$. We consider a non-backtracking $\{2^2,5,13\}$-isogeny cycle of a supersingular elliptic curve over $\Fpbar$. The details of this computation can be found in the code file \verb|example_nonprincipalcycle.ipynb|.

We have the following isogenies $\varphi_{i_{(\cdot)}}$, where $\deg\varphi_{i_{(\cdot)}} = i$:
\begin{multicols}{2}
\begin{enumerate}
    \item $\varphi_{2,1}:E_0\to E_1$,
    \item $\varphi_{5}:E_1\to E_2$,
    \item $\varphi_{2,2}:E_2\to E_3$,
    \item $\varphi_{13}:E_3\to E_0$,
\end{enumerate}
\end{multicols}
where $E_0: y^2 = x^3 + 2236x + 1886$, $E_1: y^2 = x^3 + 732x + 2243$, $E_2: y^2 = x^3 + 750x + 791$, $E_3: y^2 = x^3 + 1996x +1015$.

This isogeny cycle is depicted in Figure~\ref{fig:decomp1}.
This cycle is non-principal. Let $\eta:=\varphi_{13}\circ\varphi_{2,2}\circ\varphi_5\circ\varphi_{2,1}$. The endomorphism $\eta$ contains the factor $[2]$, since $\eta(E_0[2]) = \mathcal{O}_{E_0}$. However, the isogeny cycle does not contain backtracking. This is not a contradiction of Corollary~\ref{prop:cyclickernel}, because the isogeny cycle has not been presented in canonical form.

To refactor $\eta$ into canonical form, we swap $\varphi_5$ and $\varphi_{2,2}$ in the order of compositions. Define a new $\psi_{2,2}:E_1\to E_0$ with $\ker\psi_{2,2} = \ker\widehat{\varphi_{2,1}}$. In fact, $\psi_{2,2} = \pm\widehat{\varphi_{2,1}}$, and here we see the backtracking. 
Next, compute $\psi_{2,2}(\ker\varphi_5) =\widehat{\varphi_{2,1}}(\ker\varphi_5)$ to define $\psi_5:E_0\to E_3$. This completes the refactorisation of $\eta$ as $\eta:=[2] \circ \varphi_{13}\circ\psi_5$, as $\widehat{\varphi_{2,1}}\circ\varphi_{2,1}=[2]$. See Figure~\ref{fig:decomp2}. In this decomposition, we see that $\eta$ is the concatenation of two isogeny cycles: one principal $\{5,13\}$-isogeny cycle and one non-principal $\{2^2\}$-isogeny cycle. In particular, once the prime-degree isogenies in the cycle have been sorted in order of degree, the resulting walk has backtracking, and the canonical decomposition has the scalar factor \([2]\).

\begin{figure*}[t!]
    \centering
    \begin{subfigure}[t]{0.475\textwidth}
        \centering
        \begin{tikzpicture}
\node[circle,minimum size=.8cm,draw,line width=.4mm] (E0) at (0,0) {$E_0$};
\node[circle,minimum size=.8cm,draw,line width=.4mm] (E1) at (2,0) {$E_1$};
\node[circle,minimum size=.8cm,draw,line width=.4mm] (E2) at (2,-2) {$E_2$};
\node[circle,minimum size=.8cm,draw,line width=.4mm] (E3) at (0,-2) {$E_3$};
\node[] (phi21) at (1,1) {$\varphi_{2,1}$};

\draw[->, line width=.6mm,dotted] (E0) to [out=45,in=135] (E1);
\draw[->, line width=.6mm] (E1) -- (E2) node [midway,right, fill=white] {$\varphi_{5}$};
\draw[->, line width=.6mm,dotted] (E2) -- (E3) node [midway,below, fill=white] {$\varphi_{2,2}$};
\draw[->, line width=.6mm,dashed] (E3) -- (E0) node [midway, left, fill=white] {$\varphi_{13}$};
\end{tikzpicture}
        \caption{The original $\{2^2,5,13\}$-isogeny cycle whose edges compose to the endomorphism $\eta$ described in Example~\ref{ex:npisocycle}.}\label{fig:decomp1}
    \end{subfigure}%
    \hfill
    \begin{subfigure}[t]{0.475\textwidth}
        \centering
        \begin{tikzpicture}
\node[circle,minimum size=.8cm,draw,line width=.4mm] (E0) at (0,0) {$E_0$};
\node[circle,minimum size=.8cm,draw,line width=.4mm] (E1) at (2,0) {$E_1$};
\node[circle,minimum size=.8cm,draw,line width=.4mm] (E2) at (2,-2) {$E_2$};
\node[circle,minimum size=.8cm,draw,line width=.4mm] (E3) at (0,-2) {$E_3$};
\node[] (phi21) at (1,1) {$\varphi_{2,1}$};
\node[] (phi21hat) at (1,-.25) {$\widehat{\varphi_{2,1}}$};
\node[] (phi13) at (-1,-1) {$\varphi_{13}$};

\draw[->, line width=.6mm,dotted] (E0) to [out=45,in=135] (E1);
\draw[->, line width=.6mm,dotted] (E1) to [out=225,in=315] (E0);
\draw[->, line width=.6mm,dashed] (E3) to [out=135,in=225] (E0);
\draw[->,line width=.6mm] (E0) -- (E3) node [midway, right, fill=white] {$\psi_5$};
\end{tikzpicture}
        \caption{The endomorphism $\eta$ refactored in canonical decomposition. $\eta$ is the composition of a principal $\{5,13\}$-isogeny cycle and the scalar multiplication $[2]$.}\label{fig:decomp2}
    \end{subfigure}
    \caption{Isogeny cycle decompositions.}\label{fig:decomp}
\end{figure*}

\end{example}

\subsection{Equivalence classes of cycles}

In undirected graphs, a graph-theoretic cycle beginning at a base point vertex $V$ can be traversed in two possible directions from that base point $V$: the reversing of direction is the effect of taking the dual isogeny. For isogeny cycles whose degrees have more than one prime factor, the dual isogeny cycle will have a distinct canonical decomposition. To count isogeny cycles in Section~\ref{sec:countingcycles}, we will really be counting endomorphisms. When $|L|>1$ and $p\equiv1\pmod{12}$, an endomorphism uniquely (up to sign) determines the canonical decomposition of an isogeny-cycle. 
We define an equivalence
relation on the set of isogeny cycles by declaring two cycles to be equivalent
if they have the same canonical decomposition.

\section{Counting cycles}\label{sec:countingcycles}
As in the previous section, we continue to assume that $p \equiv 1 \mod 12$. 
We now give two methods to count equivalence classes of principal isogeny cycles in $L$-isogeny graphs. First using Brandt matrices, and secondly using ideal relations in quadratic orders.
Implementations of these counting methods, along with the examples given are in code file \verb|cycle_counts.ipynb|.


\subsection{Counting isogeny cycles using Brandt matrices}\label{ssec:brandt}
As briefly alluded to in Section \ref{Brandtpremiliminaires}, one way of counting cycles, in the isogeny graph $\mathcal{G}(p,\ell)$, 
is by relating them to diagonal elements on the Brandt matrix $B(\ell^k)$ for \(k \geq 1\) associated to the isogeny graph in question. 
Indeed, $B(\ell)$ can be seen as the adjacency matrix of the graph $\mathcal{G}(p,\ell)$. 
One can also obtain a formula for the
total number of cycles by considering the trace of certain Brandt matrices, which
can be computed efficiently by re-expressing them as sums of modified Hurwitz class numbers (see theorems \ref{GrossBrandtToHurwitz} and \ref{GrossHurwitzSum}). 
This method is detailed thoroughly in \cite{wissam2022loops}.
One of the advantages of this method is that one can very efficiently obtain upper bounds as well as estimates for the total number of cycles in the graph.

For a prime $p$ and a set $L:=\{\ell_1,\dots,\ell_r\}$, 
we are interested in computing the number of principal isogeny cycles of degree $\ell_1^{e_1} \cdots \ell_r^{e_r}$, 
at any base point $E$, 
in the graph $\mathcal{G}(p,L)$. 
For a fixed base point $E$, these cycles are denoted 
$\operatorname{Cycles}_p(\ell_1^{e_1},\dots, \ell_r^{e_r})_E^{\operatorname{prin}}/\sim$, 
where $\sim$ is equivalence up to refactoring. 
As an extension of the above, we would also like to compute the number of principal isogeny cycles, 
at any base point $E$, of length $R$ in $\mathcal{G}(p,L)$, 
i.e., of degree $\ell_1^{e_1} \cdots \ell_r^{e_r}$ with $\sum_{i} e_i=R$. 
In the case $L = \{\ell_1,\ell_2\}$, this is very similar to~\cite[Section 4.2]{wissam2022loops}. 

Let $\{E_1,\dots,E_n\}$ be a set of representatives for the vertices of $\mathcal{G}(p,L)$. 
Then the number of principal isogeny cycles of degree $\ell_1^{e_1} \cdots \ell_r^{e_r}$, 
at base point $E_i$, in $\mathcal{G}(p,L)$ is equal to $B_{ii}(\ell_1^{e_1}\cdots\ell_r^{e_r})$ 
minus the number of non-principal cycles (endomorphisms) of $E_i$ of degree $\ell_1^{e_1}\cdots\ell_r^{e_r}$. 

For any non-principal cycle, the scalar part of its canonical decomposition must be a multiple of \([\ell_j]\) for some \(j\). 
To remove cycles whose scalar part factors through $[\ell_j]$, we can subtract $B_{ii}(\ell_1^{e_1}\cdots \ell_j^{e_j-2} \cdots \ell_r^{e_r})$ from $B_{ii}(\ell_1^{e_1}\cdots\ell_r^{e_r})$, for all $j$. Of course, one must then account for the cycles whose scalar part factors through multiple \([\ell_j]\)'s that have been subtracted more than once. 
This number can be computed using the inclusion-exclusion principle. 
In order to simplify our notation, we extend the field of definition of $B$ to include rational numbers by setting $B(m):=0$ for $m \not \in \mathbb{Z}_{\geq 1}$.   
So, for example, if $L=\{\ell_1,\ell_2\}$, then we have that the number of endomorphisms of $E_i$ of degree $\ell_1^{e_1}\ell_2^{e_2}$ involving scalar factors is given by 
$B_{ii}(\ell_1^{e_1 -2 } \ell_2^{e_2}) + B_{ii}(\ell_1^{e_1} \ell_2^{e_2-2}) - B_{ii}(\ell_1^{e_1 -2 } \ell_2^{e_2-2})$ (cf.~\cite[Lemma 4.2]{wissam2022loops}). 
Hence, in the general case of $L=\{\ell_1, \dots, \ell_r\}$, 
the number of endomorphisms of $E_i$ of degree $\ell_1^{e_1}\cdots\ell_r^{e_r}$ involving scalar factors is
\[
\sum_{\varnothing \subsetneq J \subseteq \{1,\hdots, r\} } (-1)^{|J|-1} B_{ii}\left( \frac{\prod_{i=1}^{r} \ell_i^{e_i}}{\prod_{j \in J} \ell_j^2} \right).
\]
Summing over all vertices of the graph, 
we get the following formula. 
\begin{theorem} \label{TotalSumofCyclesBrandt} 
The total number of principal isogeny cycles of degree $\ell_1^{e_1} \cdots \ell_r^{e_r}$ 
and any base point $E_i$, in $\mathcal{G}(p,L)$, is given by 
	\begin{align*}
| \operatorname{Cycles}_p(\ell_1^{e_1},\dots, \ell_r^{e_r} )^{\operatorname{prin}} /\sim| = \sum_{J \subseteq \{1,\hdots, r\}} (-1)^{|J|} \Tr\left( B\left( \frac{\prod_{i=1}^{r} \ell_i^{e_i}}{\prod_{j \in J} \ell_j^2} \right)\right)
	\end{align*} 
\end{theorem}

Thanks to Theorem \ref{GrossBrandtToHurwitz}, we can make the above formula more tractable, by relating each trace to sums of modified Hurwitz class numbers. 
We will avoid substituting the formula from Theorem \ref{GrossBrandtToHurwitz} into that of Theorem \ref{TotalSumofCyclesBrandt} to avoid writing out a very cumbersome expression. 
However, if one wanted to compute the quantity $|  \operatorname{Cycles}_p(\ell_1^{e_1},\dots, \ell_r^{e_r} )^{\operatorname{prin}}/\sim| $ precisely, this substitution could be carried out. 
As we will see later, if one was interested in finding bounds (or estimates) for the number of principal cycles, this method can be pushed further (by involving Theorem \ref{GrossHurwitzSum}).  


We will now consider a more general---and very natural notion of cycle count. 
Let $\mathcal{C}_{E_i}(L;R)$ denote the number of equivalence classes of principal isogeny cycles, 
at base point $E_i$, of length $R$ in $\mathcal{G}(p,L)$. 
By \textit{length $R$ in $\mathcal{G}(p,L)$}, we naturally mean a cycle in the graph that is composed of $R$ edges. In other words, we are not fixing certain exponents $\{e_1, \dots, e_r\}$ and only considering the isogeny cycles of degree $\ell_1^{e_1} \cdots \ell_r^{e_r}$, but rather any isogeny cycle of degree $\ell_1^{e_1} \cdots \ell_r^{e_r}$, for any $(e_1, \dots, e_r)$ satisfying $\sum_i e_i =R$. 
This can be seen as a generalisation of the \emph{bi-route number} defined in \cite{wissam2022loops}. We then have 
	\begin{equation} \label{TotalNumofCycle}
		\sum_{i=1}^{n} \mathcal{C}_{E_i}(L;R) 
		= \sum_{i=1}^{n} \sum_{e_1+\cdots+e_r=R} | \operatorname{Cycles}_p(\ell_1^{e_1},\dots, \ell_r^{e_r} )_{E_i}^{\operatorname{prin}}/\sim |.
	\end{equation}  

Let us now try to re-express the above quantity. 
Switching the order of summation in Equation (\ref{TotalNumofCycle}) yields the following result. 
\begin{theorem} \label{23Oct190156pthm}
The total number of equivalence classes of principal isogeny cycles, 
for any base point $E_i$, of length $R$ in $\mathcal{G}(p,L)$ is given by 
    \begin{equation} \label{23Oct190156p}
        \sum_{i=1}^{n} \mathcal{C}_{E_i}(L;R) 
        = \sum_{i=0}^r (-1)^i {\binom{r}{i}} \sum_{\substack{e_1, e_2, \hdots, e_r \in \mathbb{Z}_{\geq 0} \\ e_1+\cdots+e_r = R - 2i}} \Tr(B(\ell_1^{e_1}\cdots\ell_r^{e_r})). 
    \end{equation} 
\end{theorem}
We observe in Equation (\ref{23Oct190156p}), that for each $i\in\{0,...,r\}$ and each $r$-tuple of integers $(e_1,...,e_r)$ with $e_1+\cdots+e_r = R - 2i$, the binomial coefficient ${\binom{r}{i}}$ counts the number of tuples
$(f_1,...,f_r)$ there are with $f_1 + ... + f_r = R$ and such that $f_j= e_j + 2$ for
exactly $i$ of the terms and $f_j= e_j$ for the rest. For each of these tuples, when
counting cycles of degree $\ell_1^{e_1}\cdots\ell_r^{e_r}$, Theorem \ref{TotalSumofCyclesBrandt} gives us one contribution of the
form $(-1)^i \Tr(B(\ell_1^{e_1}\cdots\ell_r^{e_r}))$. 

Similarly to Theorem \ref{TotalSumofCyclesBrandt}, one can compute the quantity in Theorem \ref{23Oct190156pthm} precisely, using Theorem \ref{GrossBrandtToHurwitz}. 
If one is looking for a bound (or an approximation) of the  total number of principal cycles of length $R$, then one can go further than Equation (\ref{23Oct190156p}), by applying the methods of \cite{wissam2022loops, GKPV_Collisions}. 
Indeed, noting that $H_p(D) \le H(D)$, as in Section \ref{Brandtpremiliminaires}, and combining Theorems \ref{GrossBrandtToHurwitz} and \ref{GrossHurwitzSum}, we have an upper bound 
    \begin{equation} \label{TraceBrandtUpperBound}
        \Tr(B(\ell_1^{e_1}\cdots \ell_r^{e_r})) \le 2 \prod_{i=1}^r \frac{\ell_i^{e_i+1}-1}{\ell_i-1}
    \end{equation}
and $\Tr(B(\ell_1^{e_1}\cdots \ell_r^{e_r})) $ is in $O(\ell_1^{e_1} \cdots \ell_r^{e_r})$. 
Hence, we get the following result on the total number of cycles or length $R$. 
\begin{corollary}
    There exists a computable constant $D$, depending on $r$, 
    such that number of equivalence classes of principal isogeny cycles of length $R$ in the generalised graph $\mathcal{G}(p,L)$ is 
    bounded above by $D \binom{R+r-1}{r-1} \ell_{r}^{R} \log^2(\ell_r)$. 
\end{corollary} 
\begin{proof} 
Noting that Equation (\ref{TraceBrandtUpperBound}) trivially holds when we have fractional coefficients, we can write 
    \begin{align*}
        \sum_{i=1}^{n} \mathcal{C}_{E_i}(L;R) 
        & \le 2 \sum_{e_1+\cdots+e_r = R} \sum_{J \subseteq \{1,\hdots, r\}} 
            \frac{\left( \prod_{i=1}^{r} \ell_i^{e_i+1}\right)\big /\left( \prod_{j \in J} \ell_j^{2} \right)}{\prod_{i=1}^r\ell_i -1}\\
        & = 2 \sum_{e_1+\cdots+e_r = R}
            \prod_{i=1}^{r} \frac{\ell_i^{e_i+1}}{\ell_i -1} \sum_{J \subseteq \{1,\hdots, r\}}  \prod_{j \in J} \frac{1}{\ell_j^{2}} \\
        & = 2 \sum_{e_1+\cdots+e_r = R}
            \prod_{i=1}^{r} \frac{\ell_i^{e_i+1}}{\ell_i -1} (1+ \ell_i^{-2})\\ 
        & = 2 \sum_{e_1+\cdots+e_r = R} \prod_{i=1}^{r} \ell_i^{e_i}\left( 1+ \frac{1}{\ell_i} + \frac{2}{\ell_i (\ell_i-1)} \right)\\
        & \le 2 c_1\sum_{e_1+\cdots+e_r = R} \prod_{i=1}^{r} \ell_i^{e_i}\left( 1+ \frac{2}{\ell_i} \right), 
    \end{align*}
    where $c_1\le 4$ depends on $r$; for instance if $r\ge 6$, we can take $c_1=1$. 
    Furthermore, using known bounds on the sum of reciprocals of primes (cf. Equation (3.20) in \cite{rosser1962approximate}), 
    we have 
    \begin{align*}
        \prod_{i=1}^{r} \left( 1+ \frac{2}{\ell_i}\right) 
        & \le \exp{\left[ \sum_{i=1}^r \frac{2}{\ell_i} \right] } \\ 
        & \le \exp{\left[ 2 \left( \log \log (\ell_r) + 0.2615 + \frac{1}{\log^2(\ell_r)} \right) \right] } \\         
        & \le c_2  \log^2 (\ell_r),        
    \end{align*}
    with $c_2 \le 3$ if $r\ge 4$. Letting $D:=2 c_1 c_2$, and bounding $\prod_{i=1}^{r} \ell_i^{e_i}$ by $\ell_r^R$,  
    we obtain the desired result. 
\end{proof}

Of course, one could get a better upper bound than in Inequality (\ref{TraceBrandtUpperBound}) and thus a better upper bound for $\sum_{i=1}^{n} \mathcal{C}_{E_i}(L;R)$, depending on one's intended application.  

For a heuristic \emph{estimate}, rather than an upper bound, we can argue (as in \cite[Section 5]{GKPV_Collisions}) that we expect to have $H_p(D) \approx \frac{1}{2} H(D)$ on average. This is because the prime $p$ in the definition of $H_p(D)$ in Equation (\ref{HCNDef}) splits half the time and remains inert the other half. 
This heuristic assumption gives us the estimate 
	\begin{equation} \label{approximationTrace}
		\Tr(B(\ell_1^{e_1}\cdots \ell_r^{e_r})) \approx \prod_{i=1}^r \frac{\ell_i^{e_i+1}-1}{\ell_i-1} 
					- \frac{1}{2} \sum_{d| \ell_1^{e_1}\dots \ell_r^{e_r}} \min \left\{ d, \frac{\ell_1^{e_1}\cdots \ell_r^{e_r}}{d}\right\},
	\end{equation} 
which can then be substituted into Equation (\ref{23Oct190156p}) of Theorem \ref{23Oct190156pthm}. 
Again, we will not perform this substitution, to avoid writing a very long and cumbersome expression for the approximation of $\sum_{i=1}^{n} \mathcal{C}_{E_i}(L;R)$, but this can very easily be done (by hand or on a computer), especially once the number of primes in the set $L$ is known. 
The advantage of this approach is that it is much faster to compute the right-hand side of~\Cref{approximationTrace} than the actual trace of $B(\ell_1^{e_1}\cdots \ell_r^{e_r})$.

\subsection{An ideal interpretation of certain isogeny cycles}\label{ssec:idealcounting}

In this section, we restrict to $\{\ell_1,\dots,\ell_r\}$-isogeny cycles, where  $p \neq \ell_i$ and $\ell_i < \ell_{i+1}$ for all $i$, to give an ideal-theoretic interpretation.


As isogeny cycles correspond to endomorphisms of elliptic curves, there is a connection to the theory of embeddings into quaternion orders. 
For a prime $\ell$, the Deuring correspondence \cite{Deuring41} gives a bijection between $\ell$-isogenies of supersingular elliptic curves over $\Fpbar$ and integral left ideals of norm $\ell$, of maximal orders in $\mathcal{B}_p$, up to isomorphisms of curves/orders. 
A non-scalar element $\alpha$ of a maximal quaternion order (associated to a non-scalar endomorphism of a supersingular elliptic curve) generates an imaginary quadratic order over $\mathbb{Z}$. 
In this way, we view isogeny cycles as embeddings of imaginary quadratic orders into the maximal quaternion order. 
We count isogeny cycles by relating them to these embeddings, extending work of \cite{ColoKohel20,Onuki21,Arpin2022OrientationsAC} to the $L$-isogeny graph setting. We recall some of the results and definitions from these works briefly here. 
For the purposes of this discussion, $\mathcal{O}$ will always denote an order in an imaginary quadratic field. 

\begin{definition}[Primitive $\mathcal{O}$-embedding {\cite[Def. 3.1, 3.3]{Onuki21}}]
    Let $\mathcal{O}$ denote an imaginary quadratic order, $K:=\mathcal{O}\otimes_\mathbb{Z}\mathbb{Q}$, and let $E$ be a supersingular elliptic curve. An embedding $\iota:K\hookrightarrow\End(E)\otimes_\mathbb{Z}\mathbb{Q}$ is called a \emph{primitive} $\mathcal{O}$-embedding if $\iota(K)\cap\End(E) = \iota(\mathcal{O})$.

    When focusing on the imaginary quadratic order rather than the field, we may denote a primitive $\mathcal{O}$-embedding $\iota$ by $\iota:\mathcal{O}\hookrightarrow\End(E)$.
\end{definition}

\begin{definition}[{\cite[Sec. 3.1]{Onuki21}}]
    Let $\SSOpr$ denote the set of isomorphism classes of supersingular elliptic curves $E$ together with a choice of primitive $\mathcal{O}$-embedding into $\End(E)$.
\end{definition}

Recall that we let $h_\mathcal{O}$ denote the class number of the imaginary quadratic order $\mathcal{O}$.

\begin{proposition}\label{prop:idealsandendos}
    The class group of $\mathcal{O}$ acts freely on $\SSOpr$ and has one or two orbits, depending on if $p$ is ramified or inert (resp.) in the imaginary quadratic field containing $\mathcal{O}$. In particular, $|\SSOpr| = h_\mathcal{O}$ if $p$ is ramified and $|\SSOpr| = 2 h_\mathcal{O}$ if $p$ is inert.

    If $\mathfrak{l}_i$ are ideals above $\ell_i$ in $\mathcal{O}$ such that $\prod_{i=1}^r\mathfrak{l}_i$ is principal, then the endomorphisms of the elliptic curves in $\SSOpr$ which correspond to $\prod_{i=1}^r\mathfrak{l}_i$ have canonical decompositions which are principal $\{\ell_1,\dots,\ell_r\}$-isogeny cycles. 
\end{proposition}
\begin{proof}
    This proposition condenses several results in \cite{Onuki21,Arpin2022OrientationsAC}, and applies these results to the case of $\{\ell_1,\dots,\ell_r\}$-isogenies.
\end{proof}

\begin{definition}\label{def:Iell1ellr}
       Define the set $\mathcal{I}_{p}$ of imaginary quadratic orders $\mathcal{O}$ such that:
    \begin{itemize}
        \item $p$ is inert in $\mathcal{O}\otimes_\ZZ\QQ$ and
        \item $p$ does not divide the conductor of $\mathcal{O}$.
    \end{itemize}
\end{definition}
The first condition guarantees that $|\SSOpr| = 2h_\mathcal{O}$.

\begin{definition}
    For an imaginary quadratic order $\mathcal{O}$, let $r_\mathcal{O}(n)$ denote the number of elements of norm $n$ in $\mathcal{O}$.
\end{definition}

The set $\mathcal{I}_p$ is infinite, but for any fixed $n$, $\mathcal{I}_p$ contains finitely many orders for which $r_{\mathcal{O}}(n)$ is nonzero.
With this classification of orders in mind, we now count isogeny cycles:
\begin{theorem}\label{thm:isogenycount_ideals}
    Let  $p,\ell_1,\dots,\ell_r$ be distinct primes. 
    The number of distinct canonical decompositions of principal $\{\ell_1,\dots,\ell_r\}$-isogeny cycles in the supersingular $L$-isogeny graph over $\Fpbar$ is 
    \begin{equation}\label{eq:cyclecount_ideals}
        |\mathrm{Cycles}(\{\ell_1,\dots,\ell_r\})^{\mathrm{prin}} / \sim| = \sum_{\mathcal{O}\in \mathcal{I}_p} \frac{h_{\mathcal{O}}\cdot r_{\mathcal{O}}(\ell_1\cdots\ell_r)}{2}.
    \end{equation}

\end{theorem}

\begin{proof}
    Let $n = \prod_{i=1}^r\ell_i$. For the sake of this proof, we will consider elements $(E,\iota)\in\SSOpr$ up to complex conjugation. That is, if $\theta$ is a generator of the imaginary quadratic order $\mathcal{O}$ and $\iota(\theta) = \iota'(\overline{\theta})$ where $\overline{\theta}$ denotes the complex conjugate of $\theta$, we will consider $(E,\iota)$ to be equivalent to $(E,\iota')$. Denote these equivalence classes by $[(E,\iota)]\in(\SSOpr/\sim)$. Note that $|\SSOpr| = 2|\SSOpr/\sim|$.
    
    Let $R_\mathcal{O}(n)$ denote the set of elements of an imaginary quadratic order $\mathcal{O}$ whose norm is equal to $n$. Denote by $r_\mathcal{O}(n) = |R_\mathcal{O}(n)|$. Associated to a fixed set $\mathcal{I}_p$, we consider the set 
    \begin{equation*}
    \begin{split}
        \bigsqcup_{\mathcal{O}\in\mathcal{I}_p} (\SSOpr/\sim)&\times (R_\mathcal{O}(n)/\mathcal{O}^\times) =\\
        &\{([(E,\iota)],\{\pm\alpha\}): \mathcal{O}\in \mathcal{I}_p, [(E,\iota)]\in (\SSOpr/\sim), \alpha\in\mathcal{O}, N(\alpha)=n\}.
    \end{split}
    \end{equation*}
    The size of this set is
    \[\sum_{\mathcal{O}\in\mathcal{I}_p}|\SSOpr/\sim|\cdot \frac{r_\mathcal{O}(n)}{|\mathcal{O}^\times|} = \sum_{\mathcal{O}\in\mathcal{I}_p} \frac{h_\mathcal{O}\cdot r_\mathcal{O}(n)}{2},\]
    as $|\SSOpr/\sim| = h_\mathcal{O}$ since $p$ is inert in $\mathcal{O}\otimes_\ZZ\QQ$ and $|\mathcal{O}^\times| = 2$ since $p\equiv 1\pmod{12}$.
    
    We define a map:
    \begin{equation}\label{eq:rho_map}
    \begin{split}
        \rho: \bigsqcup_{E}(\mathrm{Cycles}(\{\ell_1,\dots,\ell_r\})^{\mathrm{prin}}_E / \sim) \to \bigsqcup_{\mathcal{O}\in\mathcal{I}_p} (\SSOpr/\sim)\times (R_\mathcal{O}(n)/\mathcal{O}^\times)      
    \end{split}
    \end{equation}
    as follows. Given an $\{\ell_1,...,\ell_r\}$-isogeny cycle, let $\alpha$ denote the endomorphism obtained by composing the isogenies in this cycle starting at a vertex with isomorphism class representative $E$. By definition, $N(\alpha) = n$ and $\alpha$ is an imaginary quadratic integer. Let $\mathcal{O}= \mathbb{Z}[\alpha]$. The element $\alpha$ determines an embedding $\iota:\mathcal{O}\hookrightarrow\End(E)$, so $p$ is inert in $\mathcal{O}\otimes_\ZZ \QQ$ and $\mathcal{O}\in\mathcal{I}_p$.  The orientation $\iota$ is determined by its image on the generator of $\mathcal{O}$, which is unique up to complex conjugation in $\mathcal{O}$. To see that $\iota$ is a primitive embedding (and thus $[(E,\iota)]\in\SSOpr$), we show that there does not exist an endomorphism $\beta\in\End(E)$ such that $\alpha=[k]\circ\beta$ for some integer $k\neq\pm 1$. Since the norm of $\alpha$ is squarefree and $N([k]\circ\beta) = k^2N(\beta)$, such a decomposition is indeed not possible. The map $\rho$ is well-defined, and we will complete the proof by showing that $\rho$ is a bijection. 

    To see that $\rho$ is surjective, consider $((E,\iota),\{\pm\alpha\})$ where $\iota:\mathcal{O}\hookrightarrow\End(E)$ is a primitive embedding. The element $\alpha\in\End(E)$ necessarily has cyclic kernel, as it has squarefree degree, and so $\alpha$ can be factored into isogenies of prime-power degree, thus determining a principal isogeny cycle in $\mathrm{Cycles}(\{\ell_1,\dots,\ell_r\})^{\mathrm{prin}} / \sim$.

    To see that $\rho$ is injective, note that an isogeny cycle determines the element $\alpha$ uniquely up to automorphism and $\Aut(E) = \{[\pm1]\} = \mathcal{O}^\times$. The map $\rho$ is thus a bijection, and by our previous count of the size of the codomain we get the formula in \Cref{eq:cyclecount_ideals}. 
\end{proof}

\subsection{Examples}\label{sec:examples}

\noindent Code for these examples is given in \verb|cycle_counts.ipynb|.

\begin{example}[Principal $\{2,3\}$-isogeny cycles for $p = 61$]\label{ex:23p61}
For a small example, we wish to count the $\{2,3\}$-isogeny cycles in the supersingular $\{2,3\}$-isogeny graph over $\overline{\mathbb{F}}_{61}$, see Figure~\ref{fig:examplecounting_2}. In this case, \emph{all} $\{2,3\}$-isogeny cycles are principal, as there is no possibility of backtracking. 

A count in the graph finds ten distinct principal $\{2,3\}$-isogeny cycles:
\[(9,9),(9,9), (\alpha,\overline{\alpha}),(\overline{\alpha},\alpha),(\alpha,\overline{\alpha}),(\overline{\alpha},\alpha),(\alpha,50),(50,\alpha),(\overline{\alpha},50),(50,\overline{\alpha}).\]

\begin{figure}
    \centering
    \begin{tikzpicture}
        \node[draw,circle] (9) at (0,0) {\tiny{$9$}};
        \node[draw, circle] (alpha) at (1,.5) {\tiny{$\alpha$}};
        \node[draw, circle] (alphabar) at (1,-.5) {\tiny{$\overline{\alpha}$}};
        \node[draw,circle] (50) at (2,0) {\tiny{$50$}};
        \node[draw,circle] (41) at (3.5,0) {\tiny{$41$}};

        \draw[->,loop left] (9) to (9);
        \draw[-] (9) to (alpha);
        \draw[-] (9) to (alphabar);
        \draw[-] (alpha) to (alphabar);
        \draw[-] (alpha) to (50);
        \draw[-] (alphabar) to (50);
        \draw[-] (50) to (41);
        \draw[->,loop, looseness=8] (41) to [out=60,in=30] (41);
        \draw[->,loop, looseness=8] (41) to [out=330,in=300] (41);

        \draw[->,dashed,loop,looseness=8] (9) to [out=120,in=90 ] (9);
        \draw[->,dashed,loop,looseness=8] (9) to [out=240,in=270 ] (9);
        \draw[-,dashed] (9) to [out=60,in=120] (41);
        \draw[-,dashed] (9) to [out=300,in=240] (41);
        \draw[-,dashed] (alpha) to [out=225,in=135] (alphabar);
        \draw[-,dashed] (alpha) to [out=315,in=45] (alphabar);
        \draw[-,dashed] (alpha) to [out=0,in=135] (50);
        \draw[-,dashed] (alphabar) to [out=0,in=225] (50);
        \draw[->,dashed,loop,looseness=8] (50) to [out=60,in=30] (50);
        \draw[->,dashed,loop,looseness=8] (50) to [out=330,in=300] (50);
        \draw[-,dashed] (alpha) to [out=30,in=135] (41);
        \draw[-,dashed] (alphabar) to [out=330,in=225] (41);
    \end{tikzpicture}
    \caption{The supersingular $\{2,3\}$-isogeny graph over $\overline{\mathbb{F}}_{61}$, with vertices labelled by $j$-invariant in $\mathbb{F}_{61^2}$. Solid lines are $2$-isogenies, dashed lines are $3$-isogenies, and $\alpha,\overline{\alpha}$ denote conjugate $j$-invariants in $\mathbb{F}_{61^2}\setminus\mathbb{F}_{61}$. Loops may only be traversed in one direction, while all other edges are undirected and may be traversed in either direction.  
    }
    \label{fig:examplecounting_2}
\end{figure}

\paragraph{Counting using ideals} First, note that $r_\mathcal{O}(6) = 0$ for any $\mathcal{O}$ of discriminant greater in absolute value than $4\cdot 6$. The discriminants of orders in $\mathcal{I}_p$ which also satisfy $|\disc(\mathcal{O})|\leq 24$ are precisely:
\[-7,-8,-11,-23,-24.\]
These are discriminants of maximal orders, in number fields with class numbers $1,1,1,3,$ and $2$, respectively. The elements of norm 6 can be found by factoring the ideals above 2 and 3 and searching for the generators of principal ideals of norm 6. The results are summarized in Table~\ref{tab:examp}. Putting these values into the formula from Theorem~\ref{thm:isogenycount_ideals}:
\[|\mathrm{Cycles}(\{2,3\})^{\mathrm{prin}} / \sim| = \frac{1\cdot 4}{2} + \frac{3\cdot 4}{2} + \frac{2\cdot 2}{2} = 10.\]

\begin{table}[h!]
\centering 
\begin{tabular}{c|c|c}
$\disc(\mathcal{O})$ & $h_\mathcal{O}$ & Elements of norm 6                                                                     \\ \hline
$-7$                 & 1               & -                                                                                      \\ \hline
$-8$                 & 1               & $\pm\left(\sqrt{-2} + 2 \right)$, $\pm\left(\sqrt{-2} - 2 \right)$                     \\ \hline
$-11$                & 1               & -                                                                                      \\ \hline
$-23$                & 3               & $\pm\left(\frac{1 + \sqrt{-23}}{2}\right)$, $\pm\left(\frac{1 - \sqrt{-23}}{2}\right)$ \\ \hline
$-24$                & 2               & $\pm\sqrt{-6}$                                                                         \\
\end{tabular}
\caption{Elements counted in the ideal-theoretic count of isogeny cycles in Example~\ref{ex:23p61}}
\label{tab:examp}
\end{table}

\paragraph{Counting using Brandt matrices:}
    Let us now count the number of $\{2,3\}$-isogeny cycles, but using the Brandt matrix method. 
    As in Section \ref{ssec:brandt}, we have that $| \operatorname{Cycles}_{61}(2,3)^{\operatorname{prin}}/\sim | = \Tr(B(2\cdot3))$. This number can be computed directly using SageMath since we are working with small numbers. 

Moreover, as discussed in Section \ref{ssec:brandt}, if we wanted to avoid computing this number directly, via Brandt matrices, we can use Theorems \ref{GrossBrandtToHurwitz} and \ref{GrossHurwitzSum}, together with the bound $H_p(D) \le H(D)$ or the estimate $H_p(D) \approx \frac{1}{2} H(D)$. 
Indeed, we can easily compute (and this remains easy to do even when the degree, given by its prime factorisation, grows larger): 
    \begin{equation}
        2 \sum_{\substack{d | 6 \\ d > 0}} d - \sum_{\substack{d | 6 \\ d>0 }} \min\{d,6/d\} = 18. 
    \end{equation}
So, we thus obtain the bound $| \operatorname{Cycles}_{61}(2,3)^{\operatorname{prin}}/\sim | \le 18$ and the estimate $| \operatorname{Cycles}_{61}(2,3)^{\operatorname{prin}}/\sim |  \approx 9$, which is not far at all from the true answer $10$. 

\end{example}

Here we provide an example of a $\{\ell_1^{e_1}, \dots, \ell_r^{e_r}\}$-isogeny cycle which does not arise from the theory of ideals discussed in Section~\ref{ssec:idealcounting}.
\begin{example}[]
    \label{example_ideal_counts_limited_exponent}
    We consider again the setup of Example $\ref{ex:23p61}$, namely we take $L=\{2,3\}, p = 61$. A subgraph of $\mathcal{G}(p, L)$ is shown in Figure $\ref{fig:counterexample}$. 
    \begin{figure}
    \centering
    \begin{tikzpicture}
        \node[draw, circle] (alpha) at (-.75,1.3) {\tiny{$\alpha$}};
        \node[draw, circle] (alphabar) at (-.75,-1.3) {\tiny{$\overline{\alpha}$}};
        \node[draw,circle] (50) at (0,0) {\tiny{$50$}};
        \node[draw,circle] (41) at (1.5,0) {\tiny{$41$}};

        \draw[->, blue] (alpha) to node[black,left,pos=0.7] {$\varphi_1$} (50);
        \draw[-] (alphabar) to (50);
        \draw[->, blue] (50) to node[black, below] {$\varphi_2$} (41);

        \draw[-,dashed] (alpha) to [out=210,in=150] (alphabar);
        \draw[<-,dashed, blue] (alpha) to [out=30,in=90] node[black, above, pos=0.6] {$\varphi_3$} (41);
        \draw[-,dashed] (alphabar) to [out=330,in=270] (41);
    \end{tikzpicture}
    \caption{A subgraph of the supersingular $\{2,3\}$-isogeny graph over $\overline{\mathbb{F}}_{61}$ (see Figure \ref{fig:examplecounting_2}).}
    \label{fig:counterexample}
    \end{figure}
    The 4 supersingular curves pictured; $E_{50}, E_{41}, E_\alpha, E_{\overline{\alpha}}$, have embeddings of the imaginary quadratic order $\mathbb{Z}[\sqrt{-11}]$, which has conductor 2. For the latter 3 curves, such an embedding will be primitive. In the case of $E_{50}$ we find that this curve has a primitive embedding of the maximal order $\mathbb{Z}[\frac{1+\sqrt{-11}}{2}]$ contained in the imaginary quadratic field $\mathbb{Q}(\sqrt{-11})$. Note these orders have class numbers 3 and 1 respectively.
    \par
    Consider the endomorphism $\varphi_3 \circ \varphi_2 \circ \varphi_1$ of degree $2^2\cdot 3$ depicted in blue. The isogeny $\varphi_3$ is horizontal with respect to the $\mathbb{Z}[\sqrt{-11}]$-embedding on $E_{41}$ and $E_\alpha$, and thus corresponds to the action of a prime ideal $\mathfrak{l}_3$ lying above $(3) \cdot \mathbb{Z}[\sqrt{-11}]$, namely $(3,1 + \sqrt{-11})$. On the other hand the 2-isogenies $\varphi_1, \varphi_2$ are ascending and descending respectively, and thus do not arise from the action of a prime ideal in this order. 
    \par
    Using degree $2^2 \cdot 3$ in our implementation of Theorem~\ref{thm:isogenycount_ideals}, yields a count of $0$ principal isogeny cycles, when the correct count is $16$.
    This is because our algorithm only counts cycles where steps are all horizontal in $L$-isogeny volcano rims, while some principal cycles such as $\varphi_3 \circ \varphi_2 \circ \varphi_1$ exist, which ascend and descend in the $2$-isogeny volcano.
\end{example}

\section{Graph cuts}\label{sec:graphcuts}

In this section we study cuts of $\ell$- and $L$-isogeny graphs with low edge expansion (see Definition~\ref{def:EdgeExpansion}). We no longer restrict to the case $p\equiv1\pmod{12}$.
For a regular graph, when taking a random edge from a random vertex within the cut, a lower edge expansion implies a higher chance the edge is internal to the cut, and lower chance it leaves the cut.
This is cryptographically relevant as most isogeny-based schemes use Ramanujan or rapid mixing properties of isogeny graphs. In particular, the distribution of endpoints of random walks approaches the discrete uniform distribution on the vertices quickly as the length of the walk increases. 
For low expansion cuts however, should they exist, the probabilities are imbalanced for vertices on opposite sides.

\subsection{Edge expansion in small cuts}\label{sec_small_clusters}

For convenience we now refer to a cut as a set of vertices $C \subset V$, where as a disjoint partition this is $V = C \sqcup (V \setminus C)$.
We define a cut to be \textit{(dis)connected} if the induced subgraph from the vertices in the cut is (dis)connected.
As we are looking for cuts with relatively low edge expansion, we only care about connected cuts by the following result.
\begin{lemma}
    Let $G = (V, E)$ be a graph and $C$ a disconnected cut such that $|C|\leq |V\setminus{C}|$. Then there exists a connected cut $T \subset C$ with $E(T, C \setminus T) = \varnothing$ and $h_G(T) \leq h_G(C)$.
\end{lemma}
\begin{proof}
    Let $C_1, C_2, \dots, C_n$ be the vertex sets of the connected components of subgraph induced by \(C\). Let $E_i$ denote the set of edges in the subgraph induced by the vertices $C_i$.
    We have $\Vol(C) = \sum \Vol(C_i)$ and $|E(C, V \setminus C)| = \sum |E(C_i, V \setminus C_i)|$.
    Suppose for contradiction that $h_G(C_i) > h_G(C)$ for all $i$.
    Then 
    \[|E(C_i, V \setminus C_i)| = \Vol(C_i) \cdot h_G(C_i) > \Vol(C_i) \cdot h_G(C).\]
    Summing over $i$ gives $|E(C, V \setminus C)| > \Vol(C) \cdot h_G(C) = |E(C,V\setminus C|)$---this is a contradiction. 
    Hence there must exist $r$ and $T = C_r$ with $h_G(T) \leq h_G(C)$.
\end{proof}

We now consider small cuts of relatively small edge expansion in $\ell$-isogeny graphs.

\begin{lemma}\label{lemma:edgexpansionl}
In a directed $\ell$-isogeny graph \(G\), take a connected cut $C$ of $n$ vertices with $n < k \cdot \log_{\ell}(p) + 1$ and define $R_{\ell, n} = \frac{n \cdot \ell + 1}{n}$.
Suppose $C$ contains no curves with additional automorphisms.
Then $h_G(C) \leq R_{\ell, n}$
with equality if and only if the induced subgraph of $C$ in $G$ is a tree (with dual edges),
and inequality implying $C$ contains a curve with a cyclic endomorphism of degree at most $p^k$.
\end{lemma}
\begin{proof}
    Within $G$ merge a maximal pairing of directed edges with their duals giving a graph with undirected edges and possibly loops which remain directed.
    It remains $(\ell+1)$-regular, directed edges being self-dual loops or edges from curves with additional automorphisms.
    We ignore the latter, as no such curves are in $C$.
    Consider the induced subgraph \(G_C\) of $C$ in $G$.
    Suppose \(G_C\) has no cycles (and so no loops). Then, it is by definition an undirected tree (since it is connected), and has $n-1$ edges.
    As the graph is $(\ell+1)$-regular, each vertex of $C$ has $(\ell+1)$ edges, so $E(C, V \setminus C) = n(\ell+1) - (n-1) = n \ell + 1$, giving $h_G(C) = R_{\ell, n}$.
    If \(G_C\) is not a tree, then it contains a (graph-theoretic) cycle, and so at least one additional edge giving $h_G(C) < R_{\ell, n}$. With $n$ vertices the cycle has degree at most $\ell^{n-1} \leq \ell^{k\log_{\ell}(p)} = p^k$,
    and composes to a cyclic endomorphism by \cite[Proposition 1]{CGL09}.
\end{proof}

One interpretation of \emph{relatively} small edge expansion is compared to connected cuts of the same number of vertices $n$.
\Cref{lemma:edgexpansionl} allows us to relate the distribution of such cuts to the previously studied distribution of $p^k$-valleys 
 \cite{BonehLove}.
 \begin{remark}
By \Cref{lemma:edgexpansionl}, small cuts with ``relatively'' small edge expansion must have edge expansion strictly less than $R_{\ell, n}$, as all these cuts have edge expansion less than or equal to $R_{\ell, n}$.
Moreover, the lemma together with \cite[Section 4]{elkies} shows for a fixed constant $k < \frac{2}{3}$ and $n < k \cdot \log_{\ell}(p) + 1$, cuts of size $n$ contain curves with endomorphisms of degree less that $p^k$.
For sufficiently large $p$ these curves are rare, appearing in small clusters which are far apart, referred to as $p^k$-valleys.
For smaller $k$ (and also $n$), the $p^k$-valleys become sparser \cite{BonehLove}.
 \end{remark}

 If we try the same argument for $L$-isogeny graphs, we run into an issue, due to the existence of graph-theoretic cycles which compose to scalar endomorphisms (see~\Cref{fig:isogenydiamond}). To maximize expansion, we seek to minimize cycles in the subgraphs induced by the cut. Trees have maximal expansion. In the $\ell$-isogeny graph, consider non-backtracking paths in subtrees of the full graph. These paths correspond to cyclic isogenies which are not endomorphisms. 
 \Cref{lemma:edgexpansionl} references the edge expansion of trees, and we need the appropriate analogue of a tree in the $L$-isogeny graph setting. In this setting, the correct structure is a connected induced subgraph of $\mathcal{G}(p,L)$ where all cycles compose to scalar endomorphisms, which we call a `\treestar.'
 
\begin{definition}[Minimum Uniform Edge Expansion $R_{L, n}$]
    Fix a set of primes primes $L=\{\ell_1,\dots, \ell_r\}$ and $n > 0$.
    For a prime $p \not\in L$ we define $R_{L, n}(p)$ as the minimum edge expansion over 
    We then define $R_{L, n} = \lim_{p \to \infty} R_{L, n}(p)$.
\end{definition}

\begin{lemma}
    Fix $L$ and $n$, as above. The limit $R_{L, n} = \lim_{p \to \infty} R_{L, n}(p)$ is well-defined, and is attained for sufficiently large \(p\).
\end{lemma}
\begin{proof}
    The length of a cycle in an induced subgraph of $\mathcal{G}(p, L)$ on $n$ vertices can be bounded by the fact there are exactly $n\cdot(\ell_i + 1)$ outgoing $\ell_i$-isogenies, up to post-composition with automorphism, originating from vertices in the subgraph.
    Hence the degree of the endomorphism formed by composing a cycle can be upper bounded by $K = \prod_i \ell_i^{n\cdot(\ell_i + 1)}$,  which is independent of $p$.
    Hence for sufficiently large $p$ we may assume $K \ll p$.
    By the results of \cite{BonehLove}, increasing $p$ increases the distance between $K$-valleys until there is a large enough space in the graph for a curve $E_p$ and depth-$(n+1)$ 
    neighbourhood of $L$-isogenies, disjoint from $K$-valleys.
    This means the neighbourhood is a \treestar, i.e. it contains no non-scalar cycles of degree less than or equal to $K$.
    The $n$-vertex cut defining $R_{L, n}(p)$ has an isomorphic cut (i.e., graph isomorphism of the induced subgraphs) within the neighbourhood of $E_p$.
    The neighbourhood of $E_p$ is isomorphic as a graph to the neighbourhood arising from any larger value of $p$.
    Take a point far enough from the $K$-valleys so that the neighbourhood in the $\ell_i$-isogeny graph looks like a tree, for each $i = 1,\hdots,n$. The cycles in the union of these $\ell_i$-isogeny trees compose to scalar endomorphisms. 
    Hence for large enough $p$, $R_{L, n} = R_{L, n}(p)$ is the minimum edge expansion of all $n$-vertex cuts in the neighbourhood of $E_p$.
\end{proof}

We now consider the analogue of Lemma \ref{lemma:edgexpansionl} in $L$-isogeny graphs.
\begin{lemma}
\label{lem:Lcyclic}
In a directed $L=\{\ell_1,\dots,\ell_r\}$-isogeny graph, take a connected cut $C$ of $n$ vertices with $n < k \cdot \log_{\ell_r}(p) + 1$.
Suppose $C$ contains no curves with additional automorphisms.
Then $h_G(C) < R_{L, n}$ implies $C$ contains a curve with a cyclic endomorphism of degree at most $p^k$.
\end{lemma}
\begin{proof}
    Follows as in the proof of Lemma~\ref{lemma:edgexpansionl}, replacing `tree' with depth-$n$ `\treestar' from the definition of $R_{L, n}$.
\end{proof}

\begin{remark}
    
    For large $p$ there are many depth-$d$ $L$-isogeny neighbourhoods containing $n$ vertices within $\mathcal{G}(p,L)$ without non-scalar cycles.
    An $n$-vertex cut $C$ has ``relatively small" edge expansion if $h_G(C)<R_{L,n}$. 
    (For the $L$-isogeny graph, it is possible for $h_G(C)> R_{L,n}$, where as $h_G(C)\leq R_{\ell,n}$ by 
    \Cref{lemma:edgexpansionl}.) By \Cref{lem:Lcyclic} along with the result of~\cite[Section 4]{elkies} implies that for $k < \frac{2}{3}$ and small enough $n$ that any $n$-vertex cut of relatively small edge expansion, intersects a $p^k$-valley, which for larger $p$ and smaller $k$ are increasingly rare and sparse.
    
\end{remark}

\subsection{Clustering based on Fiedler's vector}
We now discuss finding larger subsets of vertices which define cuts of relatively small edge expansion.
We take an experimental approach on isogeny graphs $\mathcal{G}(p, L)$ with $p$ as large as computationally feasible.
Some preliminary attempts showed for $p$ large enough, $n$-vertex sets 
with edge expansion less than $R_{L, n}$ in $\mathcal{G}(p, L)$ are rare.
Searching through all $n$ vertex subsets attempting to find one, being exponential time in $p$, is computationally infeasible.
Instead of trying them all, techniques from spectral graph theory allow us to make a sensible guess. 

While several algorithms exist using spectral methods to obtain low edge expansion cuts, we will use \emph{Fiedler's algorithm}, which we explain here.
For a graph $G = (V, E)$ with vertices denoted \(v_1, \hdots, v_k\),
it proceeds as follows:
\begin{enumerate}[noitemsep]
    \item Compute the eigenvector $\vec{x}$ of the Laplacian matrix corresponding to the second largest eigenvalue $\lambda_2$\footnote{For regular graphs this is also the eigenvector of the second largest eigenvalue of the adjacency matrix.}. 
    \item Order vertices $v_i$ with respect to the quantity $x_i$, largest first, and relabel them $u_1, \dots, u_k$ with respect to this ordering.
    \item For $i = 1, \dots, k$, let $C_i = \{u_1, \dots, u_i\}$, called a \textit{sweep cut}, and compute $h_G(C_i)$ and $h_G(V \setminus C_i)$.
    \item Return the cut $C_i$ such that $\max(h_G(C_i), h_G(V \setminus C_i))$ is minimised.
\end{enumerate}
The algorithm is polynomial time in $|V| + |E|$, which for $G = \mathcal{G}(p,L)$ means polynomial time in $p$ and $\sum \ell_i$.
The eigenvector of the second largest eigenvalue $\lambda_2$---often called the \emph{Fiedler vector}---is used as it is believed to provide a good measure of connectivity.
Its use is attributable to Fiedler \cite{fiedler1973algebraic}.
In his original work however, Fiedler suggested use of the \textit{Fiedler cut} constructed as $C = \{v_i : x_i \geq 0 \}$ for clustering.
More generally, when using multiple eigenvectors, this is \textit{hyperplane rounding}, which for a single vector amounts to picking a threshold to define the cut.
A choice of $0$ for the threshold is better suited to well-balanced graphs, while allowing it to vary can result in better clustering  of unbalanced graphs.
The concept of finding cuts to minimise edge expansion didn't arise until later, and it is now more common to use the minimal edge expansion to decide the threshold.
Also, since we desire a partition into two `inter-connected' clusters either side of the cut, we minimise $\max\{h_G(C_i), h_G(V \setminus C_i)\}$ rather than just $h_G(C_i)$, to ensure the complement also has low edge expansion.
For random graphs we expect the resulting cut size to be around half of the vertices.
For a summary of spectral clustering techniques using Fiedler's vector see \cite{fielderlecturenotes, Mahoney16a}.

\begin{example}
    For the isogeny graph $\mathcal{G}(p, L)$ with $p=419$ and $L=\{2,3\}$, the second largest eigenvalue of the Laplacian is $\lambda_2 \approx 11.17$.
    The above algorithm finds sweep cut $C_{18}$ has the minimal value $\max(h_G(C_{18}), h_G(V \setminus C_{18})) \approx 0.46$.
    The graph has a total of $36$ vertices, so $C_{18}$ contains exactly half.
    The code for this example in given in file \verb|fiedler.ipynb|.
\end{example}

\subsection{Other Algorithms for Computing Graph Cuts}\label{sec_orderings}

To see if the cuts resulting from this algorithm have small edge expansion, when compared to other cuts of roughly the same size, we now give an alternative non-spectral method.
To do this we replace the ordering of vertices in the previous algorithm with a different ordering.
We want the vertices to form a subgraph that looks as `complete' as possible.
The most natural way to do this is to take the neighbourhood of a vertex.
We present two variants of this idea.
\begin{description}[wide, labelindent=0pt]
    \item[\textbf{Neighbour Expansion}:] Pick a starting vertex $[v_0]$ then add its neighbours $[v_0, v_1, v_2, v_3]$ then add the neighbours of $v_1$, then the neighbours of $v_2$ etc, until all vertices are ordered $[v_0, v_1, \dots, v_n]$.
    
    \item[\textbf{Greedy Neighbour Expansion}:] Pick a starting vertex $S_0 = [v_0]$ then define $v_{i+1}$ to be the neighbour of vertices in $S_i$ such that $\phi(S_i \cup \{v_{i+1}\})$ in minimised, then $S_{i+1} := S_i \cup \{v_{i+1}\}$. Here $\phi$ of a cut denotes the maximum of the edge expansion of the cut and the edge expansion of its complement.
\end{description}
Both of these algorithms are exponential in $\log(p)$ as they have to loop over all vertices in the graph. However the second is slower as it is quadratic in the number of vertices (as is Fiedler's algorithm), while the first is linear in the number of vertices.

The optimum solution, that is, the cut $S \subseteq V$ minimising the value of $\phi(S)$, is called the \textit{Cheeger cut}.
While computationally infeasible to compute, there are known upper and lower bounds,
$$\frac{\lambda_{-2}}{2} \leq \phi(S) \leq \sqrt{2\lambda_{-2}} ,$$
where $\lambda_{-2}$ is the second smallest eigenvalue of the Laplacian.
We also compute these bounds for reference.

Results are given in Figure \ref{fig:fiedler_ordering_results}
with code in file \verb|fiedler_comparison.ipynb|.
With the spectral ordering for reference, we try both algorithms three times, staring from different random vertices, and take the average of values $\phi$ found. From the results we make the following observations:
\begin{enumerate}
\setlength\itemsep{0.1em}
    \item The spectral ordering performs badly, suggesting there is not one distinguishable low edge 
 expansion cluster of size roughly $\frac{|V|}{2}$. Perhaps there are more, and the spectral ordering is merging such cuts, or perhaps better cuts have fewer 
 vertices. 
    \item The neighbour ordering finds better cuts than the spectral ordering, suggesting neighbourhoods of vertices form smaller edge expansion clusters.
    The greedy neighbour ordering does even better, which motivates further study of the less-trivial structures that
    arise from this ordering.
    \item As the size of the set $L$ increases, the improvements reduce.
    The Cheeger lower bound also increases, and so there is less room for improvement.
    For the neighbour orderings this is likely as there are more edges, increasing the connectivity of any subset of vertices.
    This also suggests $\ell$-isogeny graphs have smaller expansion clusters, and more of a bottleneck than $L$-isogeny graphs with larger $L$.
\end{enumerate}
Figure~\ref{fig_fiedler_orderings} with code in \verb|fiedler_viz.ipynb|, gives a visual comparison of these results.
Each image represents an adjacency matrix where a non-zero entry is given by a dot.
The different images correspond to vertices in the adjacency matrix ordered in different ways.
We first order the vertices using the spectral ordering, neighbour ordering and greedy neighbour ordering.
As the resulting cuts are around half the vertices, the edges within each cut are those in the top left quadrant of each image.
The edges fully outside the cut are those in the bottom left quadrant.
The edges between the cut and its complement lie in the off-diagonal quadrants.
Hence a low edge expansion cut will have more vertices in the diagonal quadrants and less in the off-diagonal quadrants.
To gain a perspective of the density within each quadrant, we give the same images with vertices  within the cut randomly shuffled, and those in the complement also randomly shuffled.

For the spectral ordering we see this does no better than a random ordering of vertices, as all quadrants are equally full.
The neighbour ordering sees some white space in the off diagonal quadrants,
since for the first few curves we pick, their neighbours are guaranteed to also be early on in the ordering, and so there are no edges to vertices at end of the ordering.
The greedy neighbour ordering, which performs best, essentially pushing all curves which are not close to the starting curves towards the end of the ordering.

\begin{figure}[htbp!]
    \centering
    \resizebox{\textwidth}{!}{
    \begin{tabular}{||c c c c c c||} 
     \hline
     \(p, L\) & \makecell{Spectral \\ Ordering} & \makecell{Neighbour \\ Ordering} & \makecell{Greedy \\ Ordering} & \makecell{Cheeger \\ Lower Bound } & \makecell{Cheeger \\ Upper Bound } \\ [0.5ex] 
     \hline\hline
      419, \{3\} & 0.597 &  0.319 & 0.225 & 0.071 & 0.533 \\
     \hline
     419, \{2,3\} & 0.452 & 0.289 & 0.286 & 0.096 & 0.619 \\
     \hline
     419, \{2,3,5,7,11\} & 0.488 & 0.491 & 0.413 & 0.151 & 0.776  \\
     \hline
     5569, \{3\}& 0.494 & 0.301 & 0.178 & 0.068 & 0.520  \\
     \hline
     5569, \{2,3\} & 0.484 & 0.322 & 0.195 & 0.092 & 0.608  \\
     \hline
     5569, \{2,3,5,7,11\} & 0.497 & 0.473 & 0.343 & 0.275 & 1.049 \\
     \hline
     10007, \{3\} & 0.489 & 0.287 & 0.170 & 0.065 & 0.508 \\
     \hline
    \end{tabular}
    }
    \caption{Fiedler's algorithm with different vertex orderings.}
    \label{fig:fiedler_ordering_results}
\end{figure}

\begin{figure}[htbp!]
\centering
\begin{subfigure}{.32\textwidth}
    \centering
    \includegraphics[width=0.85\textwidth]{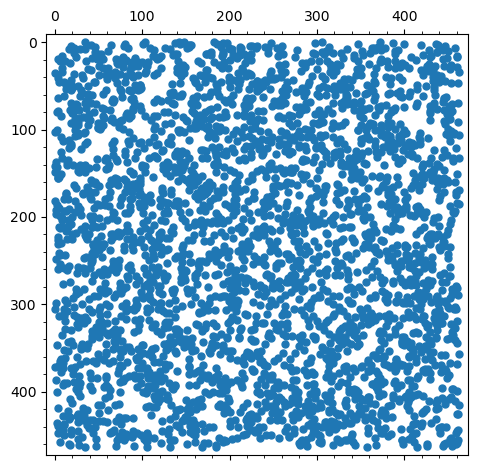}
    \subcaption{Spectral ordering}
\end{subfigure}%
\begin{subfigure}{.32\textwidth}
    \centering
    \includegraphics[width=0.85\textwidth]{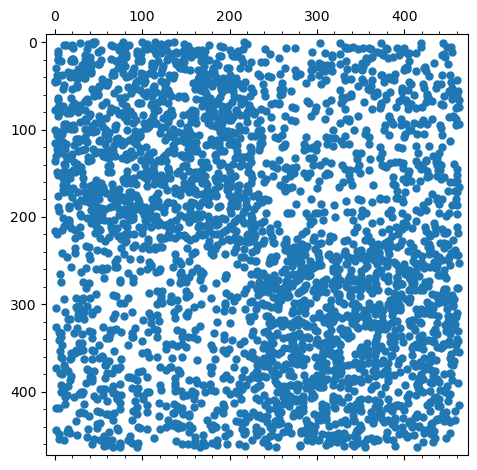}
    \subcaption{Neighbour ordering}
\end{subfigure}%
\begin{subfigure}{.32\textwidth}
    \centering
    \includegraphics[width=0.85\textwidth]{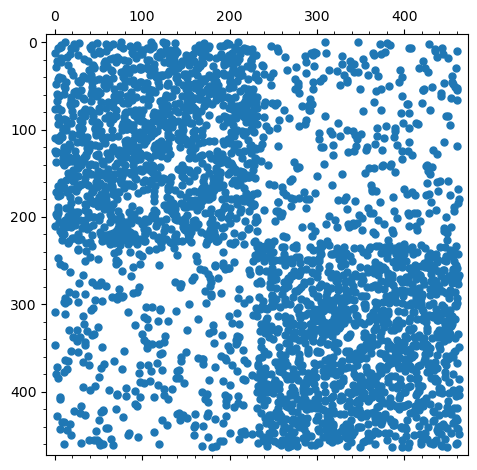}
    \subcaption{Greedy neighbour ordering}
\end{subfigure}
\caption{Adjacency matrix plot for $p=5569$ and $L=\{2,3\}$. Dots are edges. 
Vertices are ordered so that the first half of the vertices come from one side of the cut, and the second half of the vertices come from the other side.}
\label{fig_fiedler_orderings}
\end{figure}

\section{Future work}

There are several notable areas for further study. 

\subsection{Further principal cycle counting}

Firstly, within Sections~\ref{sec:cycles} and \ref{sec:countingcycles} we restricted to the case $p\equiv 1\pmod{12}$.
For other values of $p$, curves with additional automorphisms exist and so further consideration is required to determine if, and how, our principal cycle counts could be adapted to address this.

Furthermore Theorem~\ref{thm:isogenycount_ideals}, counting principal $(\ell_1^{e_1}\ell_2^{e_2}\cdots\ell_r^{e_r})$-isogeny cycles via ideals in class groups of imaginary quadratic orders, only works for products of distinct primes, with all exponents $e_i = 1$.
Example~\ref{example_ideal_counts_limited_exponent} demonstrated how the approach fails for larger exponents.
The issue arises for $e_i \geq 2$ and a starting curve with a primitive $\mathcal{O}$-embedding with $\ell_i \mid \cond(\mathcal{O})$, cycles may arise from ascending the $\ell_i$-isogeny volcano
then descending back down.
These ascending and descending isogenies are not represented by the action of invertible ideals in $cl(\mathcal{O})$.

One method to address this is to use an approach based on the volcano `rim-hopping' algorithm of \cite{ACDEKW2024}.
Given the norm $\ell_1^{e_1}\ell_2^{e_2}\cdots\ell_r^{e_r}$ generator $\eta$ of $\mathcal{O}$, 
one can obtain another generator from a translate $k + \eta$ with $k \in \mathbb{Z}$.
Then factorising its norm $N(k + \eta) = q_1^{f_1} \cdots q_m^{f_m}$ look for principal products of ideals in $\mathcal{O}$ which consist of $f_i$ ideals of norm $q_i$ for each $i$. 
If $k$ is chosen such that $q_i \nmid \cond(\mathcal{O})$, in terms of endomorphisms we obtain all degree $q_1^{f_1} \cdots q_m^{f_m}$ endomorphisms $\varphi$ which give primitive $\mathcal{O}$-embeddings.
Translating back, $\psi = \varphi - [k]$, we obtain all degree $\ell_1^{e_1}\ell_2^{e_2}\cdots\ell_r^{e_r}$ endomorphisms giving primitive $\mathcal{O}$-embeddings.
From the product of norm $q_i$ ideals we can easily check if $\varphi$ is cyclic by ensuring ideal inverses do not appear in the product. However,
further thought is needed to determine how to check, from the ideal product alone, which translated endomorphisms $\psi$ will be cyclic.
It is also possible more careful analysis of oriented isogeny volcanoes may result in an alternative approach which more closely matches Theorem~\ref{thm:isogenycount_ideals}.

\subsection{Scalar cycle counting}
Within Section~\ref{sec:countingcycles}, we restricted our counts to principal isogeny cycles, up to the equivalence of having the same canonical decomposition. Endomorphisms arising from principal cycles have the advantage that all of their decompositions yield genuine (non-backtracking) cycles. 
However, counting principal cycles is not the same as counting graph-theoretic cycles. To bridge this gap, one would have to also consider non-principal cycles, as well as cycles not up to canonical equivalence. 
The task of counting non-principal cycles additionally requires determining how many prime decompositions of a given non-cyclic endomorphism have backtracking. In particular, if we counted the number of graph-theoretic cycles which compose to scalar endomorphisms, this would be a step towards understanding the $L$-isogeny cycles generally. 
This is also closely tied to computing the minimum uniform edge expansion $R_{L, n}$ defined in Section~\ref{sec_small_clusters}.

If $[n] = [\ell_1\ell_2\cdots\ell_r]$, the total number of walks corresponding to $[n]$ is $\frac{(2r)!}{2^r}\prod_{i=1}^r(\ell_i+1)$. If $r=1$, none of these are a graph-theoretic cycle (non-backtracking). If $r=2$, precisely $2(\ell_1+1)(\ell_2+1)$ of these walks is a graph-theoretic cycle (without backtracking), namely the isogeny diamond and the same diamond traversed in the opposite direction. If $r>2$, an inclusion-exclusion argument can be applied to find the number of graph-theoretic cycles.

\subsection{Further graph clustering}

From Section~\ref{sec_orderings} the unexpectedly good performance of the greedy neighbour ordering prompts questions into the structure of this ordering, and what the consequences of this might be for sampling isogenies in cryptography.
There are also alternative graph clustering techniques, such as flow-based clustering, where the results could give better indications of graph structure.


One may also wish to consider a more extreme interpretation of small clusters; those with a single vertex. Their connectivity with the rest of the graph can be represented by a \textit{distance distribution}.
\begin{definition}[Distance distribution]
The distance distribution for a vertex $v\in V$ of a graph $G = (V,E)$ is defined as the set $(N_0(v),N_1(v),\dots)$, where $N_i(v)$ is equal to the number of vertices in the graph $G$ which are distance $i$ from the vertex $v$.    
\end{definition}
For random graphs this typically resembles a normal distribution, and the same seems to hold for isogeny graphs, as can be seen in Figure~\ref{fig_dist_distribution}.
Examining outliers and skew of these distributions could be insightful.
\begin{figure}[htbp!]
    \centering
    \includegraphics[scale=.4]{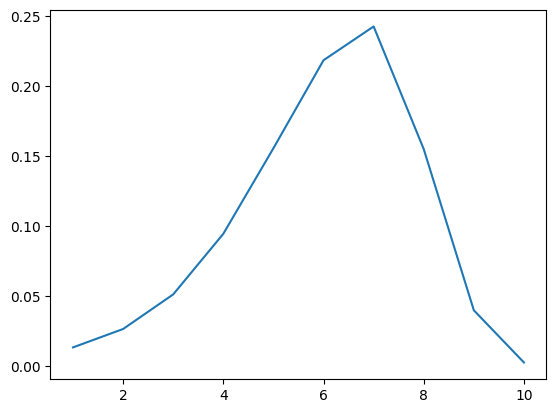}
    \caption{Distance distribution in $2$-isogeny graph with $p=2689$, from vertex with $j$-invariant $30$. Distance is denoted by $x$-axis, and the $y$-axis gives the proportion of vertices of this distance.}
    \label{fig_dist_distribution}
\end{figure}

\newpage 
\bibliographystyle{elsarticle-num}
\bibliography{biblio}\vspace{0.75in}

\end{document}